\documentclass[12pt]{article}
\usepackage{amsmath, exscale, epsfig,amssymb, 
 makeidx}

\small
\normalsize
\usepackage[latin1]{inputenc}
\usepackage[T1]{fontenc}
\usepackage{amsfonts}
\usepackage{amsbsy}
\usepackage{amscd}
\usepackage{theorem}

\usepackage{epsf}
\usepackage{psfrag}
\usepackage{graphicx}

\usepackage{multicol}
\allowdisplaybreaks

\def\build#1_#2^#3{\mathrel{\mathop{\kern 0pt#1}\limits_{#2}^{#3}}}
\def\noi{{\noindent}}

\def\cq{$\hfill \square$}
\def\un{{\bf 1}}
\newcommand{\bbD}{\mathbb{D}}
\newcommand{\bD}{\mathbb{D}}
\newcommand{\bbE}{\mathbb{E}}
\newcommand{\bE}{{\bf E}}

\newcommand{\bm}{{\bf m}}

\newcommand{\bN}{\mathbb{N}}

\newcommand{\bbP}{\mathbb{P}}

\newcommand{\bQ}{\mathbb{Q}}

\newcommand{\bP}{{\bf P}}
\newcommand{\bR}{\mathbb{R}}
\newcommand{\R}{\mathbb{R}}

\newcommand{\bbT}{\mathbb{T}}

\newcommand{\cI}{\mathcal{I}}

\newcommand{\cF}{\mathcal{F}}
\newcommand{\cG}{\mathcal{G}}
\newcommand{\cH}{\mathcal{H}}

\newcommand{\cM}{\mathcal{M}}
\newcommand{\cN}{\mathcal{N}}
\newcommand{\cP}{\mathcal{P}}

\newcommand{\cT}{\mathcal{T}}

\def\cJ{{\cal J}}
\def\cD{{\cal D}}

\def\bn{{\rm n}}

\def\varep{\varepsilon}

\def\be{\begin{equation}}
\def\ee{\end{equation}}
\def\ba{\begin{eqnarray*}}
\def\ea{\end{eqnarray*}}

\def\noi{\noindent}

\newcommand{\lgeo}{[\![}
\newcommand{\rgeo}{]\!]}

\def\qed{$\hfill \square$}
\def\btau{{\boldsymbol{\tau}}}
\def\inject{{\rm  j}}

\newtheorem{theorem}{Theorem}[section]
\newtheorem{lemma}[theorem]{Lemma}
\newtheorem{proposition}[theorem]{Proposition}

{\theorembodyfont{\rmfamily}\newtheorem{definition}{Definition}[section]}
{\theorembodyfont{\rmfamily}}
{\theorembodyfont{\rmfamily}\newtheorem{remark}{Remark}[section]}
\begin{document}

\title{ {\bf Packing and Hausdorff measures of stable trees.}}
\author{Thomas {\sc Duquesne}  
\thanks{Laboratoire de Probabilit\'es et Mod\`eles Al\'eatoires; 
Universit\'e Paris 6, Bo\^ite courrier 188, 4 place Jussieu, 75252 Paris Cedex 05, FRANCE. Email: thomas.duquesne@upmc.fr . This work benefited from the partial support of ANR-BLAN A3
} }

\vspace{2mm}
\date{\today} 

\maketitle

\begin{abstract} 
In this paper we discuss Hausdorff and packing measures of  random continuous trees called stable trees. Stable trees form a specific class of L\'evy trees (introduced by Le Gall and Le Jan in \cite{LGLJ1}) that contains Aldous's continuum random tree which corresponds to the Brownian case. We provide results for the whole stable trees and for their level sets that are the sets of points situated at a given distance from the root. We first show that there is no exact packing measure for levels sets. We also prove that non-Brownian stable trees and their level sets have no exact Hausdorff measure with regularly varying gauge function, which continues  previous results from \cite{DuLG3}.

\noindent 
{\bf AMS 2000 subject classifications}: Primary 60G57, 60J80. Secondary 28A78. \\
 \noindent   
{\bf Keywords}: {\it L\'evy trees, stable trees, mass measure, local time measure, Hausdorff measure, packing measure.}
\end{abstract}

\section{Introduction}
\label{introsec}
  Stable trees are particular instances of L\'evy trees that form a class of random compact metric spaces introduced by Le Gall and Le Jan in \cite{LGLJ1} as the genealogy of Continuous State Branching Processes (CSBP for short). The class of stable trees contains Aldous's continuum random tree that corresponds to the Brownian case (see \cite{Al1, Al2}). Stable trees (and more generally L\'evy trees) are the scaling limit of Galton-Watson trees (see \cite{DuLG} Chapter 2 and \cite{Du2}). Various geometric and distributional properties of L\'evy trees (and of stable trees, consequently) have been studied in \cite{DuLG2} and in Weill \cite{Weill}. An alternative construction of L\'evy trees is discussed in \cite{DuWi1}. Stable trees have been also studied in connection with fragmentation processes: see Miermont \cite{Mier03, Mier05}, Haas and Miermont \cite{HaaMi}, Goldschmidt and Haas \cite{GolHaa} for the stable cases and see Abraham and Delmas \cite{AbDel08} for related models concerning more general L\'evy trees.

   Fractal properties of stable trees have been discussed in \cite{DuLG2} and \cite{DuLG3}: Hausdorff and packing dimensions of stable trees are computed in \cite{DuLG2} and the exact Hausdorff measure of Aldous' continuum random tree is given in \cite{DuLG3}. The same paper contains partial results for the non-Brownian stable trees that suggest there is no exact Hausdorff measure in these cases. In this paper we prove there is no exact packing measure for the level sets of stable trees (including the Brownian case) and we also prove that there is no exact Hausdorff measure with regularly varying gauge function for the non-Brownian stable trees and their level sets. 
  
\medskip

      Before stating the main results of the paper, let us recall the definition of stable CSBPs and the definition of stable trees that represent the genealogy of stable CSBPs. CSBPs are time- and space-continuous analogues of Galton-Watson Markov chains. They have been introduced by Jirina \cite{Ji} and Lamperti \cite{La2} as the 
$[0, \infty]$-valued Feller processes that are absorbed in states $\{0\}$ and $\{ \infty\}$ and whose kernel semi-group $(p_t (x,dy); x \in [0, \infty], t \in [0, \infty))$ enjoys the branching property: $p_t (x, \cdot )*p_t (x^\prime,\cdot ) = p_t(x+ x^\prime , \cdot)$, for every $x,x^\prime\in [0, \infty]$ and every $t \in [0, \infty)$. As pointed out in Lamperti \cite{La2}, CSBPs are time-changed spectrally positive L\'evy  processes. Namely, let $Y= (Y_t, t \geq 0)$ be a L\'evy process starting at $0$ that is defined on a probability space $(\Omega, \cF, \bP)$ and that has no positive jump. Let $x \in (0, \infty)$. Set $A_t= \inf \{ s \geq 0: \; \int_0^s du/ (Y_u +x) >t \}$ for any $t \geq 0$, and $T_{x}= \inf \{ s \geq 0 : \; Y_s= -x \}$, with the convention that $\inf \emptyset = \infty$. Next set $Z_t= X_{A_t \wedge T_{x}}$ if $ A_t \wedge T_{x} $ is finite and set $Z_t= \infty$ if not. Then, $Z= (Z_t, t \geq 0)$ is a CSBP with initial state $x$ (see Helland \cite{He} for a proof in the conservative cases). Recall that the distribution of $Y$ is characterized by its Laplace exponent $\psi$ given by $\bE[\exp(-\lambda Y_t)]=\exp(t\psi(\lambda))$, $t,\lambda\geq 0$ (see Bertoin \cite{Be} Chapter 7). Consequently, the law of the CSBP $Z$ is also characterised by $\psi$ and it is called its {\it branching mechanism}. 

We shall restrict to {\it $\gamma$-stable CSBPs} for which $\psi (\lambda)= \lambda^\gamma$, $ \lambda \geq 0$, where $\gamma \in (1, 2]$. The case $\gamma= 2$ shall be refered as to the {\it Brownian case} (and the corresponding CSBP is the Feller diffusion) and the cases 
$1< \gamma < 2$ shall be refered as to the 
{\it non-Brownian stable cases}. Let $Z$ be a $\gamma$-stable CSBP defined on 
$(\Omega, \cF, \bP)$. As a consequence of a result due to Silverstein  \cite{Sil68}, the kernel 
semigroup of $Z$ is characterised as follows: for any $\lambda, s,t \geq 0$, one has 
$\bE [ \exp (-\lambda Z_{t+s} ) | Z_s]= \exp (-Z_s u(t, \lambda))$, where $u(t,\lambda)$ is the 
unique nonnegative solution of $\partial u (t , \lambda) /
 \partial t =-u(t, \lambda)^\gamma $ and $u(0, \lambda)= \lambda$. This ordinary differential equation can be explicitly solved as follows. 
\begin{equation}
\label{explisolbranch} 
u(t, \lambda)= \Big( (\gamma\! -\!1) t  +\frac{_1}{^{\lambda^{\gamma -1}}}  \Big)^{-\frac{1}{\gamma-1}}\, , \quad  t, \lambda \geq 0 \; .
\end{equation}
It is easy to deduce from this formula that $\gamma$-stable CSBPs get almost surely extinct in finite time with probability one: $\bP (\exists t \geq 0: \; Z_t= 0)= 1$. We refer to Bingham \cite{Bi2} for more details on CSBPs.

L\'evy trees have been introduced by Le Gall and Le Jan in \cite{LGLJ1} via a coding function called the {\it height process} whose definition is recalled in Section \ref{Levytreesec}. Let us briefly recall the formalism discussed in \cite{DuLG2} where L\'evy trees are viewed as random variables taking values in 
the space of all compact rooted $\bR$-trees.
Informally, a {\it $\bR$-tree} is a metric space $(\cT,d)$ such that for any two points
$\sigma$ and $\sigma^\prime $
in $\cT$ there is a unique arc with endpoints $\sigma$ and $\sigma^\prime$ and this arc is isometric to a compact interval of the real line. A {\it rooted $\bR$-tree} is a $\bR$-tree with a distinguished point that we denote by $\rho$ and that we call the root. We say that two rooted $\bR$-trees are {\it equivalent} if there is a
root-preserving isometry that maps one onto the other. Instead of considering all compact rooted $\bR$-trees, we introduce the set $\bbT$ of equivalence classes of compact rooted
$\R$-trees. Evans, Pitman and Winter in \cite{EvPitWin} noticed that $\bbT$ equipped
with the Gromov-Hausdorff distance \cite{Gro}, is a Polish space (see Section \ref{Levytreesec} for more details).

With any stable exponent $\gamma \in (1, 2]$ one can associate a sigma-finite measure $\Theta_\gamma$ on $\bbT$  called the "law'' of the $\gamma$-stable tree. Although $\Theta_\gamma$ is an infinite measure, one can prove the following: Define $\Gamma (\cT)= \sup_{\sigma \in \cT} d(\rho, \sigma)$ that is the {\it total height of} $\cT$. Then, for any $a \in (0, \infty)$, one has 
$$\Theta_\gamma ( \Gamma (\cT)> a )=  \left( (\gamma\! -\!1) a    \right)^{-\frac{1}{\gamma-1}}\, $$
Stable trees enjoy the so-called {\it branching property}, that obviously holds true for Galton-Watson trees. More precisely, for every $a>0$, under the probability measure $\Theta_\gamma (\, \cdot \, | \, \Gamma (\cT)>a)$ and conditionally given the part of $\cT$ below level $a$, the subtrees above level $a$ are distributed as the atoms of a Poisson point measure whose intensity is a random multiple of $\Theta_\gamma $, and the random factor is the total mass of the $a$-local time measure that is defined below (see Section \ref{Levytreesec} for a precise definition). It is important to mention that Weill in \cite{Weill} proves that the branching property characterizes L\'evy trees, and therefore stable trees.

  We now define $\Theta_\gamma$ by an approximation with Galton-Watson trees as follows. Let $\xi $ be a probability distribution on the set of nonnegative integers $\bN$. We first assume that  $\sum_{^{k\geq 0}} k\xi (k) = 1$ and  that $\xi$ {\it is in the domain of attraction of a $\gamma$-stable distribution}. More precisely, let $Y_1$ be a random variable such that $\log \bE [\exp(-\lambda Y_1) ]= \lambda^\gamma$, for any $\lambda \in [0, \infty)$. Let  $(J_k, k \geq 0)$ be an i.i.d.$\;$sequence of r.v.$\;$with law $\xi$. We assume there exists an increasing sequence $(a_p, p \geq 0)$
of positive integers such that $(a_p)^{-1}(J_1 +\cdots+J_p -p)$ converges  in distribution to $Y_1$. Denote by $\btau$
a Galton-Watson tree with offspring distribution $\xi$ that can be viewed as a random rooted 
$\R$-tree $(\btau, \delta  , \rho)$ by affecting length $1$ to each edge. Thus, $(\btau, \frac{_1}{^p}\delta  , \rho)$ is the tree $\btau$ whose edges are rescaled by a factor $1/p$ and we simply denote it by $ \frac{_1}{^p} \btau$. Then, for any $a \in (0 , \infty)$, {\it the law of $ \frac{_1}{^p} \btau$ under $\bP (\, \cdot \, | \, \frac{_1}{^p}\Gamma (\btau)  > a)$ converge weakly in $\bbT$ to the probability distribution $\Theta_\gamma (\, \cdot \, | \, \Gamma (\cT)  > a)$, when $p$ goes to $\infty$}. This result is Theorem 4.1 \cite{DuLG2}.

 Let us introduce two important kinds of measures defined 
on  $\gamma$-stable trees. Let $(\cT, d , \rho)$ be a $\gamma$-stable tree. For every $a>0$, we define the {\it $a$-level set $\cT(a)$ of $\cT$} as the set of points that are at distance $a$ from the root. Namely, 
\begin{equation}
\label{levelsetdef}
\cT(a):=\big\{\sigma\in\cT \, : \; d(\rho,\sigma)=a \; \big\} \; .
\end{equation}
We then define the random measure
$\ell^a$ on $\cT(a)$ in the following way. 
For every $\varepsilon>0$, write
$\cT_\varepsilon(a)$ for the finite subset of $\cT(a)$ consisting of
those vertices that have descendants at level $a+\varepsilon$. Then, $\Theta_\gamma$-a.e.$\;$for every
bounded continuous function $f$ on $\cT$, we have 
\begin{equation}
\label{geolocapprox}
\langle \ell^a,f \rangle=\lim_{\varepsilon\downarrow 0} \; ((\gamma \! - \!1)\varepsilon)^{\frac{1}{\gamma-1}} \!\!\! \sum_{\sigma\in \cT_\varepsilon(a)} \!\!\!  f(\sigma).
\end{equation}
The measure $\ell^a$ is a finite measure on $\cT(a)$ that is called the {\it $a$-local time measure of $\cT$}. We refer to \cite{DuLG2} Section 4.2 for the construction and the main properties of the local time measures $(\ell^a, a \geq 0)$ (see also Section \ref{Levytreesec} for more details). Theorem 4.3 \cite{DuLG2} ensures we can choose a modification of the local time measures $(\ell^a, a \geq 0)$ in such a way that $a \mapsto \ell^a$ is $\Theta_\gamma$-a.e.$\;$cadlag for the weak topology on the space of finite measures on $\cT$.

We next define the {\it mass measure} $\bm$ on the tree $\cT$ by
\begin{equation}
\label{unifmeas}
\bm=\int_0^\infty da\,\ell^a\,.
\end{equation}
The topological support of $\bm$ is $\cT$. Note that the definitions of the local time measures and of the mass measure only involve the metric properties of $\cT$.

   Let us mention that $\gamma$-stable trees enjoy the following scaling property: For any $c \in (0, \infty)$, the "law" of $(\cT , c\, d, \rho)$ under $\Theta_\gamma$ is $c^{_{^{1/(\gamma-1)}}} \Theta_\gamma$. Then, it  is easy to show that for any $a, c \in (0, \infty)$ the law of $c^{_{^{1/(\gamma-1)}}}  \langle \ell^{a/c} \rangle$ under $\Theta_\gamma$ is the law of $\langle \ell^{a} \rangle$
under $c^{_{^{1/(\gamma-1)}}} \Theta_\gamma$ (here, $\langle \ell^b \rangle $ stands for the total mass of the $b$-local time measure). Similarly, the law of $c^{_{^{\gamma/(\gamma-1)}}}  \langle \bm  \rangle$ under $\Theta_\gamma$ is the law of $\langle \bm \rangle$
under $c^{_{^{1/(\gamma-1)}}} \Theta_\gamma$. Since $\ell^a$ and $\bm$ are in some sense the most spread out measures on respectively $\cT(a)$ and $\cT$, these scaling properties give a heuristic explanation for the following results that concern the fractal dimensions of stable tree (see \cite{DuLG2} for a proof): For any $a \in (0, \infty)$, $\Theta_\gamma$-a.e.$\;$on $\{ \cT(a) \neq \emptyset\}$ the Hausdorf and the packing dimensions of $\cT(a)$ are equal to $1/ (\gamma -1)$ and $\Theta_\gamma$-a.e.$\;$the Hausdorf and the packing dimensions of $\cT$ are equal to $\gamma/ (\gamma -1)$. 

In this paper we discuss finer results concerning possible exact Hausdorff and packing measures for stable tree and their level sets. We first state a result concerning the exact packing measure for level sets. To that end, let us briefly recall the definition of packing measures. Packing measures have been introduced by Taylor and Tricot in \cite{TaTr}. Though their construction is done in Euclidian spaces, it easily extends to metric spaces and more specifically to $\gamma$-stable trees. More precisely, for any $\sigma \in \cT$ and any $r \in [0, \infty)$, let us denote by $\bar{B} (\sigma , r)$ (resp. $B(\sigma, r)$) the closed (resp. open) ball of $\cT$ with center $\sigma$ and radius $r$. Let $A \subset \cT$ and $\varepsilon \in (0, \infty )$. A {\it $\varepsilon$-packing of $A$} is a countable collection of  pairwise disjoint closed balls $\bar{B}(x_n , r_n) $,  $n \geq 0$,  such that $x_n \in A$  and $r_n \leq \varepsilon$. We restrict our attention to packing measures associated with a {\it regular gauge function} in the following sense: A function $g: (0, r_0) \rightarrow (0, \infty)$ is a regular gauge function if it is continuous, non decreasing, if $\lim_{0+} g= 0$  and if there exists a constrant $C\in (1, \infty) $ such that \begin{equation}
\label{doubling}
\exists \, C >1 \; : \quad g(2r) \leq C g(r) \;, \quad r \in (0, r_0/2) . 
\end{equation}
Such a property shall be refered as to a {\it $C$-doubling condition}. 
We then set 
\begin{equation}
\label{premeadef}
\cP^*_g (A)= \lim_{\varepsilon \downarrow 0} \; \,   \sup \Big\{ \sum_{^{n \geq 0}}  g(r_n) ;  
\;  (\bar{B}(x_n , r_n) , n \geq 0)\; \varepsilon^{_{_-}}{\rm packing} \; {\rm of} \; A  \Big\} 
\end{equation}
that is the {\it $g$-packing pre-measure of $A$} and we define the $g$-{\it packing outer measure of $A$} as 
\begin{equation}
\label{packdef}
\cP_g (A) = \inf \Big\{ \sum_{^{n \geq 0}} \cP^*_g (E_n) ; \;  A \subset \bigcup_{^{n \geq 0}} E_n  \Big\} \; .
\end{equation}
As in Euclidian spaces, $\cP_g$ is a Borel regular metric outer measure (see Section \ref{packingsec} for more details). The following theorem shows that the level sets of stable trees have no exact packing 
measure, even in the Brownian case. 
\begin{theorem}
\label{packlevel} 
Let $\gamma \in (1, 2]$ and let us consider a $\gamma$-stable tree $(\cT , d , \rho)$ under its excursion measure $\Theta_\gamma$. Let 
$g: (0, 1)\rightarrow (0, \infty)$ be any continuous function such that 
\begin{equation}
\label{levelfoninf}
\lim_{r\rightarrow 0} r^{-\frac{1}{\gamma-1}}g(r)= 0 \; .
\end{equation}
\begin{description}
\item[(i)]  \hspace{-4mm} If $\; \sum_{n \geq 1} \left[2^{\frac{n}{\gamma -1}} g (2^{-n})\right]^{\gamma} < \infty $, then for any $a\in (0, \infty)$, $\Theta_\gamma$-a.e.$\;$on $\{ \cT (a) \neq \emptyset \}$ and for $\ell^a$-almost all $\sigma$, we have 
\begin{equation}
\label{lowdensinfty}
 \liminf_{n \rightarrow \infty} \frac{\ell^a(B(\sigma , 2^{-n}))}{g(2^{-n})} = \infty \; .
\end{equation}
Moreover, if $g$ is a regular gauge function, then $\cP_g (\cT (a) \, ) = 0 $, $\Theta_\gamma$-a.e.
\item[(ii)]  \hspace{-2mm} If $\; \sum_{n \geq 1}  \left[ 2^{\frac{n}{\gamma -1}} g (2^{-n})\right]^{\gamma}  = \infty $, then for any $a\in (0, \infty)$, $\Theta_\gamma$-a.e.$\;$and for $\ell^a$-almost all $\sigma$, we have 
\begin{equation}
\label{lowdensnull}
 \liminf_{n  \rightarrow \infty} \frac{\ell^a(B(\sigma , 2^{-n}))}{g(2^{-n})} = 0\; .
 \end{equation}
Moreover, if $g$ is a regular gauge function, then $\cP_g (\cT (a) \, ) = \infty $, $\Theta_\gamma$-a.e.$\;$on the event $\{ \cT (a) \neq \emptyset \}$.  
\end{description}
\end{theorem}
This result is not surprising, even in the Brownian case, for it has been proved in \cite{LGPerTay95} that super-Brownian motion with quadratic branching mechanism has no exact packing measure in the super-critical dimension $d\geq 3$ and  \cite{LGPerTay95} provides a test that is closed in some sense to the test given in the previous theorem.

\begin{remark}
\label{packinglevelco} For any $p\geq 1$, define recursively the functions 
$\log_p$ by $\log_1= \log$ and $\log_{p+1}= \log_p \circ \log $. The previous theorem 
provides the following family of critical gauge functions for packing measures of level sets of a $\gamma$-stable tree: For any $\theta \in \bR $ and any $p \geq 1$, set 
$$ g_{p, \theta} (r)= \frac{r^{\frac{1}{\gamma-1} } }{(\log(1/r) \ldots \log_p (1/r))^{\frac{1}{\gamma}} (\log_{p+1} (1/r))^{\theta} } \; .$$
If $\gamma\theta >1$, then for any $a\in (0, \infty)$, one has $ \cP_{g_{p, \theta}} (\cT(a) \,  ) =0$, $\Theta_\gamma$-a.e.$\;$and if $\gamma \theta \leq 1$, then for any $a\in (0, \infty)$, one has $\cP_{g_{p, \theta}} (\cT (a)\, ) =\infty$, $\Theta_\gamma$-a.e.$\;$on the event $\{ \cT (a) \neq \emptyset \}$. \cq 
\end{remark}

\begin{remark}
\label{packingmass} Although the level sets of stable trees have no exact packing measure, 
the whole $\gamma$-stable tree has an exact packing measure as shown in the preprint
\cite{Du2009}. More precisely, for any $r \in (0, 1/e)$, we set 
$$ g(r) = \frac{r^{\frac{\gamma}{\gamma-1}}}{( \log \! \log 1/r )^{\frac{1}{\gamma-1}}} \;  .$$
Then, there exists $c_0 \in (0, \infty)$ such that $\cP_g = c_0  \, \bm$, $\Theta_\gamma$-a.e. \cq
\end{remark}

Let us briefly recall the definition of Hausdorff measures on a $\gamma$-stable tree $(\cT, d)$. 
Let us fix a regular gauge function $g$. For any subset $E\subset \cT$, we set ${\rm diam} (E)= \sup_{x,y\in E} d(x,y)$ that is the diameter of $E$. For any $A \subset \cT$, the {\it $g$-Hausdorff measure of $A$} 
is then given by 
\begin{equation}
\label{Hausdef}
 \cH_g (A)= \lim_{\varepsilon \downarrow 0} \; \inf \Big\{  \sum_{^{n \geq 0}} g({\rm diam} (E_n)) ; 
 \; {\rm diam} (E_n) 
 < \varepsilon  \; {\rm and} \; A \subset \bigcup_{^{n \geq 0}} E_n  \; \Big\} .
 \end{equation}
As in the Euclidian case, $\cH_g$ is a metric and Borel regular outer measure on $\cT$. In the Brownian case Theorem 1.3 in \cite{DuLG3} asserts that there exists a constant $c_1\in (0, \infty)$ such that for any $a\in (0, \infty)$, $\Theta_2$-a.e.$\;$we have $ \cH_{g_1} \left( \, \cdot\, \cap \cT (a) \, \right) = c_1\,  \ell^a $, where $ g_1(r) = r \log \log 1/r $. The non-Brownian stable cases are quite different as shown by the following proposition that asserts that in these cases, there is no exact upper-density for local time measures. Let us mention that the first point of the theorem is proved in Proposition 5.2 \cite{DuLG3}.
\begin{proposition}
\label{Htestlv} Let $\gamma \in (1, 2)$ and let $(\cT , d , \rho)$ be $\gamma$-stable tree under its excursion measure $\Theta_\gamma$. Let $g: (0, 1)\rightarrow (0, \infty)$ be a continuous function such that 
\begin{equation}
\label{levelfonsup}
 \lim_{r\rightarrow 0} g(r)= 0 \quad {\rm and} \quad \lim_{r\rightarrow 0} r^{-\frac{1}{\gamma-1}}g(r)= \infty \; .
\end{equation}
\begin{description}
\item[(i)] \hspace{-3mm}(Prop. 5.2 \cite{DuLG3})  If $\; \sum_{n \geq 1} \frac{2^{-n}}{g (2^{-n})^{\gamma-1}} < \infty $, then for any $a\in (0, \infty)$, $\Theta_\gamma $-a.e.$\;$for $\ell^a$-almost all $\sigma$, we have 
\begin{equation}
\label{uppdensnull}
 \limsup_{n \rightarrow \infty} \frac{\ell^a(B(\sigma , 2^{-n}))}{g(2^{-n})} = 0 \; .
 \end{equation}
Moreover, if $g$ is a regular gauge function, then $\cH_g (\cT (a) \, ) = \infty  $, 
$\Theta_\gamma $-a.e.$\;$on the event $\{ \cT (a) \neq \emptyset \} $. 

\item[(ii)] \hspace{-3mm} If $\; \sum_{n \geq 1} \frac{2^{-n}}{g (2^{-n})^{\gamma-1}} = \infty $, then for any $a\in (0, \infty)$,  $\Theta_\gamma $-a.e.$\;$on the event $\{ \cT (a) \neq \emptyset \} $ and for $\ell^a$-almost all $\sigma$, we have 
\begin{equation}
\label{uppdensinfty}
 \limsup_{n \rightarrow \infty} \frac{\ell^a(B(\sigma , 2^{-n}))}{g(2^{-n})} = \infty \; .
 \end{equation}
\end{description}
\end{proposition}

   Recall that a function $g$ {\it is regularly varying at $0$ with exponent $q$} iff for any $c \in (0, \infty)$, $g(cr)/g(r)$ tend to $c^q$ when $r$ goes to $0$. 
\begin{theorem}
\label{noreguHauslv} Let $\gamma \in (1, 2)$ and let $(\cT , d , \rho)$ be $\gamma$-stable tree under its excursion measure $\Theta_\gamma$. Then the level sets of $\cT$ have no exact Hausdorff measure with continuous regularly varying gauge function. More precisely, 
let $g: (0, 1) \rightarrow (0, \infty) $ be a regular gauge function that is regularly varying at $0$. 
\begin{itemize}
\item[-] Either for any $a\in (0,\infty)$, we $\Theta_\gamma$-a.e.$\;$have $\cH_g (\cT (a) )= \infty$, on $\{ \cT (a) \neq 
\emptyset \}$, 
\item[-] or  for any $a\in (0,\infty)$, we $\Theta_\gamma$-a.e.$\;$have $\cH_g (\cT (a) )= 0$.
\end{itemize}
\end{theorem}

\begin{remark}
\label{notestlv}
Proposition \ref{Htestlv} and Theorem \ref{noreguHauslv} suggest that if 
$$\sum_{n \geq 1} \frac{2^{-n}}{
g (2^{-n})^{_{\gamma-1}}} = \infty \; , $$
then, $\cH_g (\cT(a))= 0$, $\Theta_\gamma$-a.e.$\;$as conjectured in \cite{DuLG3}. The best result in this direction is 
Theorem 1.5 in \cite{DuLG3} that shows that  $\cH_g (\cT(a))= 0$, $\Theta_\gamma$-a.e.$\;$if $g$ is of the following form: 
$$ g(r)= r^{-\frac{1}{\gamma-1}} (\log \frac{_1}{^r})^{\frac{1}{\gamma-1}} (\log \! \log  \frac{_1}{^r})^u \; , $$
with $u <0$. \cq 
\end{remark}

Let us discuss now the Hausdorff properties of whole stable trees. In the Brownian case, Theorem 1.1 in \cite{DuLG3} asserts that there exists a constant $c_2\in (0, \infty)$ such that $\Theta_2$-a.e.$\; $we have $\cH_{g_2}  = c_2\, \bm$, where $g_2(r) = r^2\log  \log 1/r $. In the non-Brownian stable cases, the situation is quite different as shown by following proposition that asserts that in these  cases, the mass measure has no exact upper-density. Let us mention that the first point of the theorem is proved in Proposition 5.1 \cite{DuLG3}. 
\begin{proposition}
\label{Htestmass} Let $\gamma \in (1, 2)$ and let $(\cT , d , \rho)$ be $\gamma$-stable tree under its excursion measure $\Theta_\gamma$. Let $g: (0, 1)\rightarrow (0, \infty)$ be a function such that 
\begin{equation}
\label{massfonsup}
 \lim_{r\rightarrow 0} g(r)= 0 \quad {\rm and} \quad \lim_{r\rightarrow 0} r^{-\frac{\gamma}{\gamma-1}}g(r)= \infty \; .
\end{equation}
\begin{description}
\item[(i)] \hspace{-3mm}(Prop. 5.1 \cite{DuLG3}) If $\; \sum_{n \geq 1} \frac{2^{-\gamma \, n}}{g (2^{-n})^{\gamma-1}} < \infty $, then $\Theta_\gamma $-a.e.$\;$for $\bm$-almost all $\sigma$, we have  
\begin{equation}
\label{massuppdensnull}
 \limsup_{n \rightarrow \infty} \frac{\bm(B(\sigma , 2^{-n}))}{g(2^{-n})} = 0 \; .
 \end{equation}
Moreover, if $g$ is a regular gauge function, then $\cH_g (\cT) =  \infty  $, $\Theta_\gamma $-a.e.
\item[(ii)] \hspace{-3mm} If $\; \sum_{n \geq 1} \frac{2^{-\gamma \, n}}{g (2^{-n})^{\gamma-1}} = \infty $, then $\Theta_\gamma $-a.e.$\;$for $\bm$-almost all $\sigma$, we have 
\begin{equation}
\label{massuppdensinfty}
 \limsup_{n \rightarrow \infty} \frac{\bm(B(\sigma , 2^{-n}))}{g(2^{-n})} = \infty \; .
 \end{equation}
\end{description}
\end{proposition}
The previous proposition is completed by the following result. 
 \begin{theorem}
\label{noreguHausmass} Let $\gamma \in (1, 2)$ and let $(\cT , d , \rho)$ be $\gamma$-stable tree under its excursion measure $\Theta_\gamma$. Then $\cT$ has no exact Hausdorff measure with continuous regularly varying gauge function. More precisely, let 
$g: (0, 1) \rightarrow (0, \infty) $ be a regular gauge function that is regularly varying at $0$. 
\begin{itemize}
\item[-] Either $\cH_g (\cT) = \infty$, $\Theta_\gamma$-a.e.
\item[-] or $\cH_g (\cT) =0$, $\Theta_\gamma$-a.e.
\end{itemize}
\end{theorem}

\begin{remark}
\label{notestmass}
Proposition \ref{Htestmass} and Theorem \ref{noreguHausmass} suggest that if 
$$\sum_{n \geq 1} \frac{2^{-\gamma \, n}}{ g (2^{-n})^{_{\gamma-1}}} = \infty \, $$ 
then, $\cH_g (\cT)= 0$, $\Theta_\gamma$-a.e.$\;$as conjectured in \cite{DuLG3}. The best result in this direction is Theorem 1.4 in \cite{DuLG3} that show that  $\cH_g (\cT)= 0$, $\Theta_\gamma$-a.e.$\;$if $g$ is of the following form: 
$$ g(r)= r^{-\frac{\gamma}{\gamma-1}} (\log \frac{_1}{^r})^{\frac{1}{\gamma-1}} (\log \! \log  \frac{_1}{^r})^u \; , $$
with $u <0$. \cq 
\end{remark}

\bigskip

The paper is organised as follows. In Section \ref{packingsec}, we recall the basic comparison results on Hausdorff and packing measures in metric spaces. In Section \ref{Levytreesec}, we introduce the $\gamma$-stable height processes and the $\gamma$-stable trees, and we recall a key decomposition of stable trees according the ancestral line of a randomly chosen vertex that is used to prove the upper- and lower-density results for the local time measures and the mass measure. In Section \ref{estimsec}, we state various estimates that are used in the proof sections. Section \ref{Pfsec} is devoted to the proofs of the main results of the paper.

\section{Notation, definitions and preliminary results.}
\label{notadefsec}

\subsection{Hausdorff and packing measures on metric spaces.}
\label{packingsec}
Though standard in Euclidian spaces (see Taylor and Tricot \cite{TaTr}), packing measures are less usual in Polish spaces that is why we briefly recall few results in this section.  As already mentioned, we restrict our attention to {\it continuous gauge functions that satisfy a doubling condition}: Let us fix $C >1$. 
We denote by $\cG_C$ the set of such regular gauge functions that satisfy a $C$-doubling condition and we set $\cG= \bigcup_{C>1} \cG_C$ that is the set of the gauge functions we shall consider. Let us mention that, instead of regular gauge functions, some authors speak of {\it blanketed Hausdorff functions} after Larman \cite{Lar67}. 

  Let  $(\cT, d)$ be an uncountable complete and separable metric space. Let us fix $g\in \cG_C$. 
Recall from (\ref{packdef}) the definition of the $g$-packing measure $\cP_g$ and from (\ref{Hausdef}) the definition of the $g$-Hausdorff measure $\cH_g$. We shall use the following comparison results. 
\begin{lemma}
\label{genedens} (Taylor and Tricot  \cite{TaTr}, Edgar \cite{Edgar07}). 
Let $g \in \cG_C$. Then, for any finite Borel measure $\mu$ on $\cT$ and for any Borel subset $A$ of $\cT$, the following holds true.
\begin{description} 
\item[(i)] \hspace{-3mm}If  $\;  \liminf_{r \rightarrow 0} \frac{\mu (B(\sigma, r))}{g(r )} \leq 1 $, for any $\sigma \in A$, then $  
\cP_g (A)  \geq C^{-2} \mu (A) $. 
\item[(ii)] \hspace{-4mm} If  $\; \liminf_{r \rightarrow 0} \frac{\mu (B(\sigma, r))}{g(r)} \geq 1 $, for any $\sigma \in A$, then 
$  \cP_g (A)  \leq    \mu (A) $. 
\item[(iii)]  \hspace{-4mm} If  $ \; \limsup_{r \rightarrow 0} \frac{\mu (B(\sigma, r))}{g(r )} \leq 1 $, for any $\sigma \in A$, then $  
\cH_g (A)  \geq C^{-1} \mu (A) $. 
\item[(iv)]  \hspace{-4mm} If  $\;  \limsup_{r \rightarrow 0} \frac{\mu (B(\sigma, r))}{g(r)} \geq 1 $, for any $\sigma \in A$, then 
$  \cH_g (A)  \leq  C  \mu (A) $. 
\end{description}
\end{lemma}
Points $(iii)$ and $(iv)$ in Euclidian spaces are stated in Lemmas 2 and 3 in Rogers and Taylor \cite{RoTa}. Points $(i)$ and $(ii)$ in Euclidian spaces can be found in Theorem 5.4 in Taylor and Tricot \cite{TaTr}. We refer to Edgar \cite{Edgar07} for a proof of Lemma \ref{genedens} for general metric spaces: For $(i)$ and $(ii)$, see Theorem 4.15 \cite{Edgar07} in combination with Proposition 4.24 \cite{Edgar07}. For $(iii)$ and $(iv)$, see Theorem 5.9 \cite{Edgar07}.  
\begin{remark}
\label{packferm}  Our definition of $g$-packing measures (that is Edgar's definition in \cite{Edgar07} Section 5) is slightly different from that of Taylor and Tricot in \cite{TaTr} who use open balls packing and $g({\rm diam} (\cdot))$ as set function. However, since the gauge function is continuous and since it satisfies a doubling condition, the corresponding packing measure is equivalent to ours.  It only change the bounds in $(iii)$ and $(iv)$ in a obvious way. \cq
\end{remark}
\subsection{Height processes and L\'evy trees.}
\label{Levytreesec}
In this section we recall (mostly from \cite{DuLG} and \cite{DuLG2}) various results concerning stable height processes and stable trees that are used in Sections \ref{estimsec} and in the proof sections. 

\subparagraph{The height process.}
We fix $\gamma \in (1, 2]$. It is convenient to work on the canonical space $\bD ([0, \infty), \bR )$ of cadlag paths equipped with Skorohod distance and the corresponding Borel sigma field. We denote by $X= (X_t , t \geq 0)$ the canonical process and by $\bbP$ the canonical distribution of a $\gamma$-stable and spectrally positive L\'evy process with Laplace exponent $\psi(\lambda)= \lambda^\gamma$. Namely, $\bbE [\exp (-\lambda X_t)  ] = \exp (t \lambda^\gamma )$, for any $\lambda, t \geq 0$. Note that  $X_t$ is integrable and that $\bbE [X_t]= 0$, which easily implies that $X$ oscillates when $t$ goes to infinity. $\bbP$-a.s.$\;$the path $X$ has infinite variation (for more details, see Bertoin \cite{Be} Chapters VII and VIII ).

  In the more general context of spectrally positive L\'evy processes, it has been proved in Le Gall and Le Jan \cite{LGLJ1} and in \cite{DuLG} Chapter 1 that there exists a continuous process $H= (H_t , t \geq 0)$ such that for any $t \geq 0$, the following limit holds in $\bbP$-probability.  
\begin{equation}
\label{Hlimit}
H_t:=\lim_{\varepsilon\to 0} \frac{1}{\varep}\int_0^t {\bf 1}_{\{I^s_t<X_s<I^s_t+\varep\}}\,ds \; , 
\end{equation}
where $I^s_t$ stands for $\inf_{s\leq r\leq t} X_r$. The process $H= (H_t, t \geq 0)$ is called the $\gamma$-{\it stable height process}. As we see below, $H$ provides a way to explore the genealogy of a $\gamma$-stable CSBP. We refer to Le Gall and Le Jan \cite{LGLJ1} for a careful explanation of (\ref{Hlimit}) in the discrete setting. 

  For any $c \in (0, \infty)$, it is easy to prove that $(c^{_{^{-1/\gamma }}}X_{ct}, t \geq 0)$ has the same law as $X$ and we easily derive from (\ref{Hlimit}) that, under $\bbP$, one has 
\begin{equation}
\label{Hscaling}
\big(   c^{-\frac{\gamma -1}{\gamma}} H_{ct} , \, t \geq 0 \big) \overset{{\rm (law)}}{=} (H_t, \, t \geq 0) \; .
\end{equation}

\subparagraph{Excursions of the height process.}
In the Brownian case $\gamma= 2$,  $X$ is distributed as a Brownian motion and (\ref{Hlimit}) easily implies that $H$ is proportional to $X-I$, which is distributed as a reflected Brownian motion. In more general cases, $H$ is neither a Markov process nor a martingale. However it is possible to develop an excursion theory for $H$ as follows.  Recall that $X$ has infinite  variation sample paths. Basic results on fluctuation theory (see \cite{Be} Chapter VI.1 and VII.1) entail that $X-I$ is a strong Markov process in $[0, \infty)$ and that $0$ is regular for 
$(0, \infty)$ and recurrent with repect to this Markov process. Moreover, $-I$ is a local time at $0$ for $X-I$ (see Theorem VII.1 \cite{Be}). Denote by $N_\gamma$ the corresponding excursion 
measure of $X-I$ above $0$ and denote by $(a_j, b_j)$, $j\in  \cI$, the excursion intervals of $X-I$ above $0$ and by $X^j = X_{(a_j + \cdot )\wedge b_j}-I_{a_j}$, $j\in \cI$, the corresponding excursions.  Then, the point measure $\sum_{j\in \cI} \delta_{(-I_{a_j}, X^j)}$ is a Poisson point measure on $[0, \infty)\times \bbD([0, \infty), \bR)$ with intensity 
$dx \otimes N_\gamma (dX)$. Now, observe that (\ref{Hlimit}) implies that the value of $H_t$ only depends of the excursion of $X-I$ straddling $t$ and that $\bigcup_{^{j\in \cI}} (a_j, b_j)= \{ t \geq 0: H_t >0 \}$. 
This allows to define 
$H$ under $N_\gamma$ (see the comments in Section 3.2 \cite{DuLG} for more details). We use the slightly abusive notation $N_\gamma (dH)$ for the "distribution" of $H$ under the excursion measure $N_\gamma $ of $X-I$ above $0$. For any $j\in  \cI$, we set $H^j=H_{(a_j+\cdot )\wedge b_j} $. Then the $H^j$s are the excursions of $H$ above $0$, and the point measure 
\begin{equation}
\label{Poissheight}
\sum_{j\in \cI} \delta_{(-I_{a_j},H^j)}
\end{equation}
is distributed under $\bbP$ as Poisson point measure on $[0, \infty)\times \bD([0,\infty), \bR)$ with intensity $dx \otimes N_\gamma(dH)$. 

  Set $\zeta:=\inf\{t>0:X_t=0\}$ that is the total duration of $X$ under $N_\gamma$. Since $X$ does not drift to $\infty$, the lifetime $\zeta$ is finite $N_\gamma$-a.e. Moreover, $N_\gamma$-a.e.$\; H_0=H_\zeta =0$ and $H_t >0$ for any $t \in (0, \zeta)$. We easily deduce from (\ref{Hscaling}) the following scaling property for $H$ under $N_\gamma$: For any $c \in (0, \infty)$ and for any  measurable function $F: \bbD ([0, \infty), \bR) \rightarrow [0, \infty)$, one has 
\begin{equation}
\label{Hexcscaling}
c^{\frac{1}{\gamma}}N_\gamma \big( \, F( c^{-\frac{\gamma-1}{\gamma}} H_{ct}, \, t \geq 0  )\,   \big)= N_\gamma \big(\,  F(  H_{t}, \, t \geq 0  ) \,  \big) \; .
\end{equation}

\subparagraph{Local times of the height process.} 

We recall here from \cite{DuLG} Chapter 1 Section 1.3 the following result: There exists a jointly measurable process $(L^{_a}_{^s}, a, s\geq 0)$ such that $\bbP$-a.s.$\;$for any $a\geq 0 $, $s \rightarrow L^{_a}_{^s}$ is continuous and non-decreasing and such that 
 \begin{equation}
 \label{approxtpsloc}
 \forall t,  a \geq 0 , \quad  \lim_{\varepsilon \rightarrow 0} \bbE \left[  \sup_{ 0\leq s \leq t} \left| \frac{1}{\varepsilon} \int_0^s dr 
\un_{\{ a< H_r \leq a+\varepsilon \}} -L_s^a \right|  \, \right] =0\; . 
\end{equation}
The process $(L^{_a}_{^s}, s \geq 0)$ is called the {\it $a$-local time of $H$}. Recall that $I$ stands for the infinimum process of $X$. Then, the following properties of local-times of $H$ hold true: First observe that  $L_{^t}^{_0}= -I_t$, $t \geq 0$. Second, the support of the random Stieltjes measure $dL^{_a}_{^{\cdot }}$ is contained in the closed set $\{t \geq 0:H_t=a\}$. Moreover, a general version of the Ray-Knight theorem for $H$ asserts the following: For any $x\geq 0$, set $T_x=\inf \{ t\geq 0 \; :\;
X_t=-x\}$. Then, {\it the process $(L^a_{T_x} \; ;\; a \geq 0)$ is a distributed as a $\gamma$-stable CSBP with initial state $x$}. We refer to Le Gall and Le Jan \cite{LGLJ1} Theorem 4.2 or to \cite{DuLG} Theorem 1.4.1 for a proof of this general version of Ray-Knight Theorem. 

 The CSBP $(L^a_{T_x} \; ;\; a\geq 0)$ admits a cadlag modification that is denoted in the same way to simplify notation. An easy argument deduced from the approximation (\ref{approxtpsloc}) entails that 
$\int_0^a L^b_{T_x}  \, db= \int_0^{T_x}\un_{\{ H_t \leq a\}} dt $. This remark combined with an elementary formula on CSBPs (whose proof can be found in Le Gall \cite{LG99}) entails that  
\begin{equation}
 \label{kappadef}
 \bbE \Big[ \exp \Big(  \! -\!\mu L^a_{T_x} -\lambda \!\! \int_0^{T_x} \un_{\{ H_t \leq a\} } dt  \, \Big) \Big]= \exp \big(- x \kappa_a (\lambda , \mu)\,  \big) \; ,  \quad a, \lambda , \mu \geq 0 ,
\end{equation}   
where $\kappa_a (\lambda , \mu)$ is the unique solution of the following differential equation 
\begin{equation}
\label{equakappa}
\kappa_0 (\lambda, \mu) = \lambda  \quad {\rm and} \quad   \frac{\partial \kappa_a}{\partial a} (\lambda ,\mu ) = \lambda -\kappa_a (\lambda, \mu)^\gamma  \; , \quad a , \lambda , \mu \geq 0 .
\end{equation}
The function $\kappa$ plays an important role and we shall turn back to it further. 

\medskip 

   It is possible to define the local times of $H$ under the excursion measure $N_\gamma$ as follows. For any $b >0$, let us  set $v(b)=N_\gamma ( \sup_{^{t \in [0, \zeta ]}} H_t > b )$. The continuity of $H$ and the Poisson decomposition (\ref{Poissheight}) obviously imply that $v(b) < \infty$, for any $b >0$. 
 It is moreover clear that $v$ is non-increasing and $\lim_{\infty} v= 0 $. For every $a\in (0, \infty)$, we then define a continuous increasing process $(L^a_t,t\in [0, \zeta] )$, such that for every
$b \in (0,\infty)$ and for any $t\geq 0$, one has 
\begin{equation}
\label{localapprox}
\lim_{\varepsilon \rightarrow 0} \, 
N_\gamma \left(  \un_{\{\sup H>b \}}\;\sup_{ 0\leq s \leq t\wedge \zeta} \left| \frac{1}{\varepsilon} \int_0^s
dr 
\un_{\{ a-\varepsilon< H_r \leq a\}} -L_s^a \right| \right) =0.
\end{equation}
See \cite{DuLG} Section 1.3 for more details. The process $(L^a_t,t\in [0, \zeta] )$ is the $a$-local time of the height process. Note that $N_\gamma$-a.e.$\; $the support of the Stieltjes measure $dL^{_a}_{^{\, \cdot}}$ is contained in $\{t:H_t=a\}$.

  Recall notation $(a_j, b_j)$, $j \in \cI$, for the excursion intervals of $H$ above $0$ and set $\zeta_j= b_j-a_j$ that is the the total duration of the excursion $H^j$. One easily deduces from (\ref{localapprox}) that 
$ \mu L^{_a}_{^{T_x}} + \lambda \int_{^0}^{_{{T_x}}}\un_{\{H_t \leq a \}} dt = \sum \mu  (L^{_a}_{^{b_j}}-L^{_a}_{^{a_j}} ) + \lambda 
\int_{^0}^{_{{\zeta_{^j}}}} \un_{\{ H_t^j \leq a \} } dt $, where the sum in the right member is taken over the set indices $j \in \cI$ such that $-I_{a_j} \leq x$. Therefore, (\ref{Poissheight}) entails that 
\begin{equation}
\label{kappaexc}
N_\gamma \left( 1- e^{ -\mu L^a_\zeta - \lambda \int_0^a \un_{\{ H_t \leq a \}} dt} \right) =
 \kappa_a (\lambda, \mu) \; , \quad a , \lambda , \mu , \geq 0 . 
\end{equation}
We refer to \cite{DuLG} Chapter 1 for more details. By taking $\lambda= 0$ in the previous display, we get $N_\gamma( 1- \exp (-\mu L^{_a}_{^\zeta} )\, ) = u(a, \mu)$, where $u$ is given by (\ref{explisolbranch}). This easily entails
\begin{equation}
\label{meanloc}
\forall a \geq 0 \; , \quad N_\gamma (L^a_\zeta )= 1 \; .
\end{equation}
Let us also mention from \cite{DuLG} the following formula  
\begin{equation}
\label{vvvequa}
\forall a>0 \; , \quad v(a)= N_\gamma \big( \sup H_t \geq a \big) = N_\gamma \left( L^a_\zeta \neq 0\right) =  \big( \, (\gamma \!- \!1)a \big)^{-\frac{1}{\gamma -1}} \; .    
 \end{equation}

\subparagraph{L\'evy trees.} 

We first define $\bR$-trees (or real trees) that are metric spaces that generalise graph-trees. 
\begin{definition}
\label{errtreedef} Let $(T, \delta)$ be a metric space. It is a {\it real tree} iff the following holds true for any $\sigma_1, \sigma_1\in T$. 
\begin{description}
\item[{\bf (a)}] \hspace{-4mm} There is a unique isometry 
$f_{\sigma_1,\sigma_2}$ from $[0,\delta(\sigma_1,\sigma_2)]$ into $T$ such
that $f_{\sigma_1,\sigma_2}(0)=\sigma_1$ and $f_{\sigma_1,\sigma_2}(
\delta(\sigma_1,\sigma_2))=\sigma_2$. We denote by  $\lgeo \sigma_1,\sigma_2\rgeo$ the geodesic   joining $\sigma_1$ to $\sigma_2$. Namely, $\lgeo \sigma_1,\sigma_2\rgeo:=f_{\sigma_1,\sigma_2}([0,\delta (\sigma_1,\sigma_2)])$

\medskip

\item[{\bf (b)}] \hspace{-4mm}  If $\inject$ is a continuous injective map from $[0,1]$ into
$T$, such that $\inject(0)=\sigma_1$ and $\inject(1)=\sigma_2$, then  we have
$\inject([0,1])=\lgeo \sigma_1,\sigma_2\rgeo$.

\medskip
\end{description}
A rooted $\bR$-tree is an $\bR$-tree $(T,\delta)$ with a distinguished point ${\rm r}$ called the root.  \qed 
\end{definition}

 Among metric spaces, $\bR$-trees are characterized by the so-called {\it four points inequality} that is expressed as follows. Let $(T, \delta)$ be a connected metric space. Then, $(T, \delta)$ is a $\bR$-tree iff for any $\sigma_1,  \sigma_2,  \sigma_3,  \sigma_4  \in T$, we have 
\begin{equation}
\label{fourpoint}
\delta(\sigma_1, \sigma_2) + \delta(\sigma_3, \sigma_4) \leq \big(\delta(\sigma_1, \sigma_3) + \delta(\sigma_2, \sigma_4)\big) \vee  \big( \delta(\sigma_1, \sigma_4) + \delta(\sigma_2, \sigma_3)  \big) . 
\end{equation}
We refer to Evans \cite{EvStF} or to Dress, Moulton and Terhalle \cite{DMT96} for a detailed account on this property. The set of all compact rooted $\bR$-trees can be equipped with the pointed Gromov-Hausdorff distance in the following way. Let $(T_{1}, \delta_1, {\rm r}_1)$ and $(T_{2}, \delta_2, {\rm r}_1)$ be two compact pointed metric spaces. They can be compared one with each other thanks to 
the pointed Gromov-Hausdorff distance defined by 
$$d_{{\bf GH}}(T_1,T_2)= \inf  \;  \; \delta_{{\bf H}}\big( \, \inject_1(T_1),\inject_2 (T_2) \, \big) \vee \delta \big( \, \inject_1 ({\rm r}_1), \inject_2 ({\rm  r}_2) \, \big) \; . $$  
Here the infimum is taken over all $(\inject_1,\inject_2, (E, \delta))$, where $(E, \delta) $ is a metric space, where $\inject_1 : T_1 \rightarrow E$ and $\inject_2 : T_2 \rightarrow E$ are isometrical embeddings and where $\delta_{{\bf H}}$ stands for the usual Hausdorff metric on compact subsets of $(E, \delta)$. Obviously $d_{{\bf GH}}(T_1, T_2)$ only depends on
the isometry classes of $T_1$ and $T_2$ that map ${\rm r_1}$ to ${\rm r_2}$. In \cite{Gro}, Gromov proves that $d_{{\bf GH}}$ is a metric on the set of the equivalence classes of pointed compact metric spaces that makes it a complete and separable metric space. Let us denote by $\bbT$, the set of all 
equivalence classes of rooted compact real-trees. Evans, Pitman and Winter observed in \cite{EvPitWin} that $\bbT$ is $d_{{\bf GH}}$-closed. Therefore, $(\bbT, d_{{\bf GH}})$ is a complete separable metric space (see Theorem 2 of \cite{EvPitWin}).

  Let us briefly recall how $\bR$-trees can be obtained via continuous functions. We consider a continuous function $h:[0,\infty)\rightarrow \bR$ such that there exists $a \in [0, \infty)$ such that $h$ is constant on $[a, \infty )$. We denote by $\zeta_h$ the least of such real numbers $a$ and we view $\zeta_h$ as the lifetime of $h$. Such a continuous function is said to be a {\it coding function}.    
To avoid trivialities, we also assume that $h$ is not constant. Then, for every $s,t\geq 0$, we set
\begin{equation}
\label{pseudometric}
b_h(s,t)=\inf_{r\in[s\wedge t,s\vee t]}h(r) \quad {\rm and} \quad d_h(s,t)=h(s)+h(t)-2b_h(s,t).
\end{equation}
Clearly $d_h(s,t)=d_h(t,s)$. It is easy to check that $d_h$ satisfies the four points inequality, which implies that $d_h$ is a pseudo-metric.  We then introduce the equivalence relation
$s\sim_h t$ iff $d_h(s,t)=0$ (or equivalently iff $h(s)=h(t)=b_h(s,t)$) and we denote by $T_h$ the quotient set $[0,\zeta_h]/ \sim_h$. Standard arguments imply that $d_h$ induces a metric on $T_h$
that is also denoted by $d_h$ to simplify notation. We denote by
$p_h:[0,\zeta_h]\rightarrow T_h$ the canonical projection. Since $h$ is continuous, $p_h$ is a  continuous function from $[0,\zeta_h ]$ equipped with the usual metric onto $(T_h, d_h) $. This implies that $(T_h, d_h)$ is a compact and connected metric space that satisfies the four points inequality. It is therefore a compact $\bR$-tree. Next observe that for any $t_0, t_1\in [0, \zeta_h]$ such that $h(t_0)=h(t_1) = \min h$, we have $p_h (t_0)= p_h (t_1)$; so it makes sense to define the root of $(T_h, d_h)$ by $\rho_h= p_h (t_0)$. We shall refer to the rooted compact $\bR$-tree $(T_h, d_h, \rho_h)$ as to the {\it tree coded by $h$}.

  We next define {\it the $\gamma$-stable tree} as the tree coded by the $\gamma$-stable height process $(H_t , 0 \leq t \leq \zeta)$ under the excursion measure $N_\gamma$ and to simplify notation we set 
$$ (T_H, d_H, \rho_H )= (\cT, d , \rho) \; .$$
We also set $p= p_H: [0, \zeta] \rightarrow \cT$. Note that $\rho = p(0)$. Since $H_\zeta = 0$ and since $H_t >0$, for any $t \in (0, \zeta)$, $\zeta$ is the only time $t\in [0, \zeta]$ distinct from $0$ such that $p(t)= \rho$.  

\medskip

    Let us denote by $\bar{\cT}$ the root-preserving isometry class of $(\cT, d, \rho)$. It is proved in \cite{DuLG2} that  $\bar{\cT}$ is  measurable in $(\bbT, d_{{\bf GH}})$. We then define $\Theta_\gamma$ as the "distribution" of $\bar{\cT}$ under $N_\gamma$. 
    
\begin{remark}    
We have stated the main results of the paper under $\Theta_\gamma$ because it is more natural and because $\Theta_\gamma$ has an intrinsic characterization as shown by Weill in \cite{Weill}. However,  each time we make explicit computations with stable trees, we have to work with random isometry classes of compact real trees, which causes technical problems (mostly measurability problems). To avoid these unnecessary complications during the intermediate steps of the proofs, {\it we prefer to work with the specific compact rooted real tree $(\cT, d, \rho)$ coded by the $\gamma$-stable height process $H$ under $N_\gamma$ rather than directly work under $\Theta_\gamma$}. So, we prove the results of the paper for $(\cT, d, \rho)$ under $N_\gamma $, which easily implies the same results under $\Theta_\gamma$. \qed
\end{remark}

\subparagraph{The local time measures and the mass measure on $\gamma$-stable trees.}
As above mentioned, we now work with the $\gamma$-stable tree $(\cT, d, \rho)$ coded by $H$ under the excursion measure $N_\gamma$. 
A certain number of definitions and ideas can be extended from graph-trees to real trees such as the degree of a vertex. Namely, for any $\sigma \in \cT$, we denote by $\bn (\sigma )$ the (possibly infinite) number of connected components of the open set $\cT \backslash \{ \sigma \}$. We say that $\bn (\sigma )$ is the {\it  degree} of $\sigma$. Let $\sigma$ be a vertex {\it distinct from the root}. If $\bn (\sigma )= 1$, then we say that $\sigma $ is a {\it leaf} of $\cT$; if $\bn (\sigma ) =2$, then we say that $\sigma $ is a {\it simple point}; if $\bn (\sigma )\geq 3 $, then we say that $\sigma $ is a {\it branching point} of $\cT$. If $\bn (\sigma)= \infty$, we then speak of $\sigma$ as an {\it infinite branching point}. We denote by ${\bf Lf  } (\cT)$ the set leaves of $\cT$, we denote by ${\bf Br}(\cT)$ the set of branching points of $\cT$ and we denote by ${\bf Sk} (\cT)= \cT \backslash{\bf Lf  } (\cT) $ 
the {\it skeleton} of $\cT$. Note that the closure of the skeleton is the whole tree $\overline{{\bf Sk}} (\cT)= \cT$. Let us mention that $H$ is not constant on every non-empty open subinterval of $[0, \zeta]$, $N_\gamma$-a.e. This easily entails the following characterisation of leaves in terms of the height process:  For any $t\in (0, \zeta)$, 
\begin{equation}   
\label{Hleafcarac}
 p(t) \in {\bf Lf} (\cT)  \quad  \Longleftrightarrow \quad  \forall \varepsilon  > 0 \; , \quad \inf_{s \in [t-\varepsilon , t]} \!\! H_s \quad  {\rm and} \quad  \inf_{s \in [ t,  t+\varepsilon ]} \!\! H_s  \;  < \;  H_t \;  . 
\end{equation}
Let us now define the  the mass measure and the local time measures on $\cT$: The {\it mass measure} $\bm$ is the measure induced by the Lebesgue measure $\ell$ on $[0, \zeta]$ via $p$. Namely, for any Borel set $A$ of $\cT$, we have $ \bm (A)= \ell ( p^{-1} (A)) $. We can prove that the mass measure is diffuse and its topological support is clearly $\cT$.
Moreover $\bm$ is supported by the set of leaves: 
 \begin{equation}
\label{massskel}
\bm \big( {\bf Sk} (\cT) \big) = 0 \; . 
\end{equation}
For any $a \in (0, \infty)$, the {\it $a$-local time measure} $\ell^a$ is the measure induced by $dL^a_{\cdot}$ via $p$. Namely, 
$$ \langle \ell^a , f \rangle = \int_0^\zeta dL^a_s f(p (s)) \; , $$
for any positive measurable application $f$ on $\cT$. Let us mention that the topological support of $\ell^a$ is included in  the $a$-level set $\cT (a)= \{ \sigma \in \cT: d(\rho ,\sigma)= a \}$ and note from the definition that the total mass $\langle \ell^a\rangle$ of $\ell^a$ is $L^{_a}_{^\zeta}$. Moreover, observe that $\cT(a)$ is not empty iff $\sup H \geq a$. Then, (\ref{vvvequa}) can be rewritten as follows. 
\begin{equation}
\label{treevvvequa}
\forall \, a>0 , \quad v(a)= N_\gamma \big( \cT(a) \neq \emptyset  \big) = N_\gamma  \left( \ell^a  \neq 0\right) =  \big( \, (\gamma \!- \!1)a \big)^{-\frac{1}{\gamma -1}} \; .    
 \end{equation}
As already mentioned, the $a$-local time measure $\ell^a$ can be defined in a purely metric way by (\ref{geolocapprox}) and there exists a modification of local time measures $(\ell^a, a \geq 0)$ such that 
$a \mapsto \ell^a$ is $N_\gamma$-a.e.$\;$cadlag for the weak topology on the space of finite measures on $\cT$. Except in the Brownian case, $a \mapsto \ell^a$ is not continuous and Theorem 4.7 \cite{DuLG2} asserts that there is a one-to-one correspondence between the times of discontinuity of $a \mapsto \ell^a$, the infinite branching points of $\cT$ and the jumps of the excursion $X$ of the underlying $\gamma$-stable L\'evy process. More precisely, $a$ is a time-discontinuity of $a \mapsto \ell^a$ iff there exists a unique infinite branching point $\sigma_a\in \cT(a)$ such that $\ell^{a-}= \ell^a + \lambda_a \delta_{\sigma_a}$. Moreover, a point $\sigma \in \cT$ is an infinite branching point iff there exists $t \in [0, \zeta]$ such that $p(t)= \sigma$ and $\Delta X_t >0$; if furthermore $\sigma= \sigma_a$, then $\lambda_a = \Delta X_t$. Now, observe that if $\sigma \in \cT(a)$ is an atom of $\ell^a$, the definition (\ref{geolocapprox}) of $\ell^a$ entails that $\sigma$ is an infinite branching point and that $a$ 
is a time-discontinuity of $a \mapsto \ell^a$. Thus, $\sigma= \sigma_a$. Recall that the Ray-Knight theorem for $H$ asserts that 
$a \mapsto \langle \ell^a \rangle $ is distributed as a CSBP (under its excursion measure), which has no fixed time-discontinuity. This (roughly) explains the following. 
\begin{equation}
\label{diffloc}
 \forall \, a >0  , \quad N_\gamma   \; ^{_{_-}} {\rm a.e.} \quad  \ell^a \; {\rm is} \; {\rm diffuse} .   
\end{equation}
We refer to \cite{DuLG2} for more details. 
\subparagraph{The branching property for $H$.}
We now  describe the distribution of excursions of the height process above level $b$ (or equivalently of the corresponding stable tree above level $b$). Let us fix $b\in (0, \infty)$, and denote by $(g^{_b}_{^j}, d^{_b}_{^j})$, $j\in \cI_b$, 
the connected components of the open set $\{t \geq 0:H_t>b\}$. For any $j\in \cI_b$,
denote by $H^{_{b, j}}$ the corresponding excursion of $H$ that defined by $H^{_{b, j}}_s=H_{(g^b_j+s)\wedge d^b_j}-b$, $ s\geq 0$. 

  This decomposition is interpreted in terms of the tree as follows. Recall that $\bar{B} (\rho , b)$ stands for the closed ball with center $\rho$ and radius $b$. Observe that the connected component the open set $\cT \backslash \bar{B} (\rho , b)$ are the subtrees $\tilde{\cT}_{^j}^{_{b, o}}:= p((g^{_b}_{^j}, d^{_b}_{^j}))$, $j \in \cJ_b$. The closure $\cT_{^j}^{_b}$ of $\cT_{^j}^{_{b, o}}$ is simply $\{ \sigma^{_b}_{^j} \} \cup \cT_{^j}^{_{b, o}}$, where $\sigma^{_b}_{^j} = p(g^{_b}_{^j})= p(d^{_b}_{^j})$, that is the points on the $b$-level set $\cT(b)$ at which $\cT^{_{b, o}}_{^j}$ is grafted. Observe that the rooted compact $\bR$-tree $(\cT^{_b}_{^j}, d, \sigma^{_b}_{^j})$ is isometric to the tree coded by $H^{_{^{b, j}}}_{ }$.

  We then define $\tilde{H}^b_s=H_{\tau^b_s}$, where for 
every $s\geq 0$, we have set 
$$ \tau^b_s=\inf  \big\{  t\geq 0:\int_0^t ds\,\un_{\{H_s\leq b\}}>s \big\}.$$
The process $\tilde{H}^b=(\tilde{H}^b_s, s \geq 0)$ is the {\it height process below $b$} and the rooted compact $\bR$-tree 
$(\bar{B} (\rho , b), d, \rho)$ is isometric to the tree coded by $\tilde{H}^b$. Let $\cG_b$ be the 
sigma-field generated
by $\tilde{H}^b $ augmented by the $N_\gamma$-negligible sets. From the approximation 
(\ref{localapprox}), it follows that $L^{_b}_{^\zeta}$ is measurable with respect
to $\cG_b$. We next use the following notation 
\begin{equation}
\label{condiprob}
N_\gamma^{_{_{(b)}}}= N_\gamma (\, \,  \cdot \,\,  | \, \sup H>b)
\end{equation}
that is a probability measure and we define the following point measure on $[0, \infty) \times \bbD ([0, \infty) , \bR)$:
\begin{equation}
\label{branchprop}
\cM_b =  \sum_{^{j\in \cI_b}} \delta_{( L^b_{g^b_j} ,H^{b,j}  )}
\end{equation}
{\it The branching property at level $b$} then asserts that under $N_\gamma^{_{_{(b)}}}$, conditionally given $\cG_b$, $\cM_b $ is distributed as a Poisson point measure with intensity 
$\un_{[0,L^{b}_{\zeta} ]}(x)dx\otimes  N_\gamma (dH)$. We refer to  Proposition 1.3.1 in \cite{DuLG} or the proof of Proposition 4.2.3 \cite{DuLG}. Let us mention that it is possible to rewrite intrinsically the branching property under $\Theta_\gamma$: we refer to Theorem 4.2 \cite{DuLG2} for more details.

\subparagraph{Spinal decomposition at a random time.} 
 We recall another decomposition of the height process (and therefore of the corresponding tree) that is proved in \cite{DuLG} Chapter 2 and in  \cite{DuLG2} under a more explicit form (see also \cite{DuLG4} for further applications). This decomposition is used in a crucial way in the proof of the upper- and lower-density results for the local times measures and the mass measure. Let us introduce an auxiliary probability space $(\Omega, \cF , \bP)$ that is assumed to be 
rich enough to 
carry the various independent random variables we shall need.

  Let $(U_t, t \geq 0)$ be a subordinator defined on $(\Omega, \cF, \bP)$ with initial value $U_0= 0$ and with Laplace exponent $\psi^\prime (\lambda)= \gamma \lambda^{\gamma-1}$, $\lambda \geq 0$. Let 
\begin{equation}
\label{Nstardef}
\cN^* = \sum_{j \in \cI^*} \delta_{ (  r^*_j , \, H^{*j} )  } 
\end{equation}
be a random point measure on $[0, \infty) \times \bbD ([0, \infty), \bR)$ defined on  $(\Omega, \cF, \bP)$ such that a regular version of the law of $\cN^*$ conditionally given $U$ is that of a Poisson point measure with intensity $dU_r \otimes N_\gamma (dH)$. Here $d U_r$ stands for the (random) Stieltjes measure associated with the non-decreasing path $r \mapsto U_r$. For any $ a \in (0, \infty)$, we also set 
\begin{equation}
\label{Nstaradef}
\cN^*_a =  \sum_{j \in \cI^*} \un_{[0, a]} (r^*_j) \, \delta_{ (  r^*_j , \, H^{*j} )  } . 
\end{equation}
We next consider the $\gamma$-height process $H$ under its excursion measure $N_\gamma$. For any $t \geq 0$, we set $\hat{H}^{t}:= ( H_{(t- s)_+}, s \geq 0) $ (here, $(\, \cdot )_+$ stands for the positive part function) and $\check{H}^{t}:= (H_{(t+s)\wedge \zeta} , s \geq 0)$. We also define the random point measure $\cN_t$ on  $[0, \infty) \times \bbD ([0, \infty) , \bR)$ by 
\begin{equation}
\label{spinaldef}
\cN_t =\cN( \hat{H}^{t})+ \cN ( \check{H}^{t}) := \sum_{j \in \cJ_t } \delta_{(r^t_j, H^{*\, t,j} )} \; \, ,  
\end{equation}
where for any continuous function $h: [0, \infty ) \rightarrow [0, \infty)$ with compact support, the point measure $\cN (h)$ is defined as follows: Set $\underline{h} (t)= \inf_{[0, t]} h$ and denote by $(g_i, d_i)$, $i\in \cI(h)$ the excursion intervals of $h-\underline{h} $ away from $0$ that are the connected component of the open set $\{ t \geq 0: h(t)-\underline{h} (t) >0 \}$. For any $i \in \cI (h)$, set 
$h^i (s)  = ((h- \underline{h}) ( (g_i +s) \wedge d_i) \, , \, s \geq 0)$. We then define $\cN(h)$ as the point measure on $[0, \infty) \times \bbD ([0, \infty) , \bR)$ given by 
$$ \cN (h)= \sum_{ i\in \cI (h)} \delta_{(h(g_i) , h^i) } \; .$$
Lemma 3.4 in \cite{DuLG2} asserts the following. For any $a$ and for any nonnegative measurable function $F$ on the set of positive measures on $[0, \infty) \times \bbD ([0, \infty), \bR)$ (equipped with the topology of the vague convergence), one has 
\begin{equation}      
\label{ancdecomp}
N_\gamma \left( \int_{0}^\zeta \!\!\! dL^a_t  \; F \big(  \cN_t  \big)  \right) =  \bE \left[ F (\cN^*_a ) \right] . 
\end{equation}
We shall refer to this identity as to the {\it spinal decomposition of $H$ at a random time}.

\medskip

Let us briefly interpret this decomposition in terms of the $\gamma$-stable tree $\cT$ coded by $H$. Choose $t \in (0, \zeta)$ and set $\sigma = p(t) \in \cT$.  Then the geodesic $\lgeo \rho, \sigma \rgeo$ is interpreted as the ancestral line of $\sigma$. Let us denote by $\cT_{^j}^{_o}$, $j \in \cJ$, the connected components of the open set $\cT   \backslash \lgeo \rho , \sigma \rgeo$ and denote by $\cT_{^{_j}}$ the closure of 
$\cT_{^j}^{_o}$. Then, there exists a point $\sigma_j \in \lgeo \rho , \sigma \rgeo$ such that $\cT_{^{_j}} = \{ \sigma_{^{_j}} \} \cup \cT_{^j}^{_o}$. Recall notation $(r^{t}_{^j} , H^{_{^{*\, t,j}}})$, $j \in \cJ_t$ from (\ref{spinaldef}). The specific coding of $\cT$ by $H$ entails that for any $j \in \cJ$ there exists a unique $j^\prime \in \cJ_{t}$ such that $ d(\rho ,\sigma_{^{_j}})= r^{t}_{^{j^\prime}} $ and such that the rooted compact $\bR$-tree $(\cT_{^{_j}}, d, \sigma_{^{_j}})$ is isometric to the tree coded by $H^{_{^{*\, t,j^\prime}}}$

\medskip

 We now compute $\bm (\bar{B}(p(t), r))$ in terms of  $\cN_t$ as follows. First, recall from (\ref{pseudometric}) the definition of $b(s,t) $ and $d(s, t)$. Note that if $H_s= b(s,t)$ with $s \neq t$, then $p(s) \in {\bf Sk} (\cT)$ by (\ref{Hleafcarac}). Let us fix a radius $r$ in $(0, H_t)$. Then, (\ref{massskel}) entails 
$$ \bm \big( \bar{B} (p(t) , r  ) \big) = \int_{0}^\zeta \un_{\{ d(s,t) \leq  r \}} ds = \int_0^\zeta 
\un_{\{ 0 < H_s  - b(s,t ) \leq r -  H_t+b(s,t)     \}} . $$
The definition of  $(\cN(\hat{H}^t), \cN(\check{H}^t))$ entails 
\begin{equation}
\label{massballfixt}   
\bm \big( \overline{B}(p(t) , r  ) \big) = \sum_{j \in \cJ_t} \un_{[ H_t-r \, , \,  H_t]} (r^t_j)\cdot  \int_0^{\zeta^t_j} \un_{\{ H^{*\, t,j}_s \leq r-H_t +r^t_j\}} \, ds , 
\end{equation}
where $\zeta^t_j$ stands for the lifetime of the path $H^{*\,  t,j}$. For any $a \in (0, \infty)$ and for any $r \in [0,a]$, we next set 
\begin{equation}
\label{Mstardef}
M^*_r (a) = \sum_{j \in \cI^*} \un_{[a-r, a]} (r^*_j) \cdot \int_{0}^{\zeta^*_j} \un_{\{ H^{*j}_s \leq r-a+r^*_j  \}} ds  \; , 
 \end{equation}    
where $\zeta^*_j$ stands for the lifetime of the path $H^{*j}$. Then, $(M^*_r (a) , r \in [0 ,a])$ is a cadlag increasing process defined on $(\Omega, \cF, \bP)$. The spinal decomposition (\ref{ancdecomp}) implies that for any $a \in (0, \infty)$ and for any bounded measurable $F: \bbD ([0, a ] , \bR) \rightarrow \bR$, we have 
$$ N_\gamma \left( \int_0^\zeta \!\!\! dL^a_t  \, F \! \left(  \big( \bm( \bar{B}(p(t) , r)) \, \big)_{r \in [0, a]}  \right) \right)  \
= \bE \left[  F \big( (M^*_r (a) )_{r \in [0, a ] } \big)\right] .$$
Since the $a$-local time measure is the image measure of $dL^a_\cdot $ by the canonical projection $p$, we get 
\begin{equation}
\label{keymass}
N_\gamma \left( \int_{\cT} \!\!\! \ell^a(d\sigma) \, F  \! \left(  \, \big(  \bm ( \bar{B}(\sigma , r)) \, \big)_{r \in [0, a]}  \right)  \right)=  
\bE \left[  F \big( \, (M^*_r (a))_{r \in [0, a ] } \big) \right] . 
\end{equation}
This identity is used in the proof of Proposition \ref{Htestmass}.

  Let us discuss a similar formula for the $a$-local time measure: Let $t \in [0, \zeta ]$ be such that $H_t= a$. Namely $p(t)\in \cT(a)$. We fix $r \in (0, 2a)$. Then observe that for any $s \in [0,\zeta ]$ such that  $H_s = a $, we have $d(s, t) \leq r$ iff $b(s, t) \geq a-(r/2)$. We then get 
\begin{equation}
\label{linloc}
 \ell^a \big( \overline{B}(p(t) , r  ) \big) = \ell^a (\{ p(t) \} ) + \sum_{j \in \cJ_t} \un_{[ a-\frac{r}{2} \, , \, a)} (r^t_j) \, 
L^{a-r^t_j}_{\zeta^t_j} (t,j)  \; , 
\end{equation}
where $L^{_{a-r^t_j}}_{^{\zeta^t_j}} (t,j) $ stands for the local time at level $a-r^t_j$ of the excursion $H^{* t, j}$. Next, for any $a \geq 0$ and any $ r \in (0, 2a)$, we set 
\begin{equation}
\label{Lstardef}
L^*_r (a) = \sum_{j \in \cI^*} \un_{[a-\frac{r}{2} \, , \,  a]} (r^*_j) \, L_{\zeta^*_j}^{a-r^*_j}   ds  \; , 
\end{equation}  
where, $L_{^{\zeta^*_j}}^{_{a-r^*_j}}$ stands for the local time at level $a-r^*_j$ of the excursion $H^{*  j}$. 

  Now (\ref{diffloc}) entails that $\ell^a (\{ p(t) \} ) =0$, $N_\gamma$-a.e.$\;$and 
(\ref{linloc}) combined with the spinal decomposition (\ref{ancdecomp}) implies that for any $a \in (0, \infty)$ and for any bounded measurable $F: \bbD ([0, a ] , \bR) \rightarrow \bR$, we have 
\begin{equation}
\label{keyloc}
N_\gamma \left( \int_{\cT} \!\!\! \ell^a(d\sigma) \, F  \! \left(  \, \big(  \ell^a ( \bar{B}(\sigma , r)) \, 
\big)_{r \in [0, 2a]}  \right)  \right)=  
\bE \left[  F \big( \, (L^*_r (a) )_{r \in [0,2a ] } \big) \right]
\end{equation}
This identity is used to prove Theorem \ref{packlevel} and Proposition \ref{Htestlv}.

\subsection{Estimates.}
\label{estimsec}
Let us fix $ a>0$. Recall the definition of the $a$-local time measure $\ell^a$ (whose total mass $\langle \ell^a\rangle $ is equal to $L^a_\zeta$) and recall that 
\begin{equation}
\label{recallbranch}
N_\gamma \big(1-e^{-\lambda \langle \ell^a\rangle } \big)=  N_\gamma \big(1-e^{-\lambda L^a_\zeta } \big)= u(a, \lambda) =  \left( (\gamma\! -\!1) a  +\frac{_1}{^{\lambda^{\gamma -1}}}  \right)^{-\frac{1}{\gamma-1}}\, .
\end{equation}
Next recall from (\ref{condiprob}) the definition of $N^{_{(a)}}_\gamma $. We easily deduce from (\ref{vvvequa}) and (\ref{recallbranch}) that 
\begin{equation}
\label{Zolotalaw}
N^{_{(a)}}_\gamma\! \big( \exp (- \lambda  \langle \ell^a\rangle ) \, \big)= 1- \Big( \frac{( \gamma \! -\! 1) a \lambda^{\gamma -1}}{1+ ( \gamma \! -\!  1) a \lambda^{\gamma -1} }\Big)^{\frac{1}{\gamma-1}} \; .
\end{equation}
Consequently, 
\begin{equation}
\label{scalingZolo}
a^{-\frac{1}{\gamma -1}} \langle \ell^a \rangle \; \,  {\rm under} \; N^{_{(a)}}_\gamma \; \overset{{\rm (law)}}{=} \; \langle \ell^1 \rangle   \; \,  {\rm under} \; N^{_{(1)}}_\gamma \; .
\end{equation}
\begin{lemma}
\label{taillevdirect} For any $\gamma \in (1, 2]$, we have 
$$ N^{_{(1)}}_\gamma \!  \big(  \langle \ell_1  \rangle  \leq  \; x \,  \big) \; \sim_{x \rightarrow 0+} \;  \frac{x^{\gamma -1}}{(\gamma\!-\!1)^2 \Gamma (\gamma) }  \; .$$
\end{lemma}
\begin{proof}
From (\ref{Zolotalaw}), we get 
$$ N^{_{(1)}}_\gamma \!  \big( \exp ( -\lambda \langle \ell_1 \rangle ) \, \big) 
\;  \sim_{\lambda \rightarrow \infty} \, \frac{ \lambda^{-(\gamma -1)}}{(\gamma-1)^2} \; .$$
The desired result is then a direct consequence of a Tauberian theorem due to Feller: see \cite{Fe} Chapter XIII § 5 (see also \cite{BiGoTe} Theorem 1.7.1', p. 38 ). \qed
\end{proof}

  Recall the notation $\cN^*$ and the definition of $L^*_r (a)$ from (\ref{Lstardef}). For any $0 \leq r^\prime \leq r \leq 2a $, we set 
\begin{equation}
\label{shelldef}
 \Lambda_{r^\prime , r} (a)=\sum_{^{j \in \cI^*}} \un_{[a-\frac{r}{2} \, , \, a-\frac{r^{_\prime}}{2} \, )} \! (r^*_j) \; L_{^{\zeta^*_j}}^{_{a-r^*_j}}     \; .
\end{equation}
Observe that 
\begin{equation}
\label{increLstar}
\forall \, 0 \leq r^\prime \leq r \leq 2a \; , \quad L^*_r (a)= \Lambda_{r^\prime , r} (a) + L^*_{r^\prime} (a) . 
\end{equation}
\begin{lemma}
\label{couronne} Let $(r_n, n \geq 0)$ be a sequence such that $0 <r_{n+1} \leq r_n\leq 2a$ and $\lim_n r_n = 0$. Then, the random variables $(\Lambda_{r_{n+1}, r_n } (a), n \geq 0)$ are independent and 
\begin{equation}
\label{coucr}
L^*_{r_0} (a)= \sum_{n \geq 0} \Lambda_{r_{n+1}, r_n } (a) \; .
\end{equation}
\end{lemma}
\begin{proof} 
First, note that (\ref{coucr}) is a direct consequence of the definitions of $\Lambda_{r^\prime , r} (a)$ and of $L^*_r (a)$. Let us prove the independence property. Recall that conditionally given $U$, $\cN^*$ is a Poisson point process with intensity $dU_t \otimes N_\gamma (dH)$.  Elementary properties of Poisson point processes and the definition of the $\Lambda_{r_{n+1}, r_n } (a)$s entail that the random variables $(\Lambda_{r_{n+1}, r_n } (a),n \geq 0)$ are independent conditionally given $U$. Moreover, the conditional distribution of $\Lambda_{r_{n+1}, r_n } (a)$ given $U$ only involves the increments of $U$ on $[r_{n+1}, r_n]$, which easily implies the desired result since $U$ is a subordinator.  \qed 
\end{proof}

\begin{remark}
\label{utile} The previous lemma and (\ref{increLstar}) imply that for any $0 \leq r^\prime \leq r \leq 2a$, one has $L^*_r (a) \geq  \Lambda_{r^\prime , r} (a)$ and that $ \Lambda_{r^\prime , r} (a)$ is independent of $L^*_{r^\prime} (a)$. Observe also that the process $r \mapsto L^*_r (a)$ has independent increments. \cq 
\end{remark}

\begin{lemma}
\label{Lambstar} For any $0 \leq r^\prime \leq r \leq 2a$, we have 
\begin{equation}
\label{Lapcour}
\bE \big[ \exp (-\lambda  \Lambda_{r^\prime , r} (a) ) \big]=  \left(  \frac{\frac{_{\gamma-1}}{^2} r^\prime  \lambda^{\gamma -1}+ 1}{\frac{_{\gamma-1}}{^2} r  \lambda^{\gamma -1}+ 1} 
 \right)^{\frac{\gamma }{\gamma -1}} \; .
\end{equation}
Consequently, we get 
$$ r^{-\frac{1}{\gamma-1}}   \Lambda_{r^\prime , r} (a)  \overset{{\rm (law)}}{=}   \Lambda_{\frac{r^\prime}{r} , 1} (1)  \; .$$ 
\end{lemma}
\begin{proof}
First observe that the second point is an immediate consequence of the first one. Recall that conditionally given $U$, $\cN^*$ is distributed as a Poisson point process with intensity $dU_t \otimes N_\gamma$. Therefore, 
$$ \bE \left[ \exp \big(- \lambda  \Lambda_{r^\prime , r} (a)  \big) \, | \, U \, \right]= \exp \Big(- \int_{[a-r/2, a-r^\prime /2\, )} \!\!\!\!\!\!\! \!\!\!\!\!\!\! \!\!\!\!\!\!\! \!\!\!\!\!\!\!  dU_{^{_t}} \,  N_{^{_{_\gamma}}} \big( 1-e^{-\lambda L^{a-t}_\zeta } \big) \Big) . $$ 
Recall that $u(a-t, \lambda) = N_\gamma \big( 1-e^{-\lambda L^{a-t}_\zeta } \big) $, where $u$ is given by (\ref{explisolbranch}) and recall that $U$ is a subordinator with Laplace exponent $\lambda \mapsto \gamma \lambda^{\gamma -1}$. Thus, 
$$ \bE \left[ \exp \big(- \lambda  \Lambda_{r^\prime , r} (a)  \big) \right]= \exp \Big(- \gamma  \int_{a-r/2}^{ a-r^\prime /2\, }  \!\!\!\!\!\!\!  u(a-t, \lambda)^{\gamma -1}dt   \Big) , $$ 
which entails the desired result thanks to a simple change of variable. \qed 
\end{proof}

\medskip

\noi
Taking $r^\prime= 0$ in the previous lemma entails the following. 
\begin{lemma}
\label{Lstar} For any $a \in (0, \infty)$ and for any $r \in [0, 2a]$, we have 
$$ \bE \big[ \exp (-\lambda L^*_r (a) ) \big]= \big( 1+ \frac{_{\gamma-1}}{^2}  r \lambda^{\gamma -1}  \big)^{-\frac{\gamma}{\gamma-1}} \; .$$
Then, $ r^{{-1/(\gamma-1)}}L^*_r (a)$ has the same law as $L^*_1 (1) $. 
\end{lemma}
To simplify notation, we set 
\begin{equation}
\label{Zprimedef}
Z_\gamma := L^*_{ 1} (1)  \quad {\rm and} \quad Z^\prime_\gamma := \Lambda_{\frac{1}{2} , 1} \; .
\end{equation}
\begin{proposition}
\label{levelestimates} We have the following estimates. 
\begin{itemize} 
\item{(i)} For $\gamma \in (1, 2)$, we have 
$$ \lim_{x \rightarrow \infty} x^{\gamma -1} \bP (Z_\gamma \geq x ) =  2 \lim_{x \rightarrow \infty} x^{\gamma -1} \bP (Z'_\gamma   \geq x ) =  \frac{\gamma}{2\Gamma (2-\gamma) } .$$
\item{(ii)} For any $\gamma \in (1, 2]$ we get 
$$   \lim_{x \rightarrow 0+} x^{-\gamma} \bP (Z_\gamma \leq x )  = \frac{2^{\frac{\gamma}{\gamma -1} } }{ (\gamma -1)^{\frac{\gamma}{\gamma -1} } \Gamma (1+ \gamma) } \;. $$
\end{itemize}
\end{proposition}
\begin{proof}
First assume that $\gamma \in (1, 2)$. When $\lambda$ goes to $0$, we have
$$ \bE \left[ e^{-\lambda Z_\gamma } \right] = 1 - \frac{_\gamma}{^2} \lambda^{\gamma -1} + o(\lambda^{\gamma -1} ) \quad {\rm and} \quad \bE \left[ e^{-\lambda Z^\prime_\gamma } \right] = 1 - \frac{_\gamma}{^4} \lambda^{\gamma -1} + o(\lambda^{\gamma -1} ) \; .$$
A Tauberian theorem due to Bingham and Doney \cite{BiDo} (see also \cite{BiGoTe} Theorem 8.1.6, p. 333 ) implies $(i)$. 
Let us prove $(ii)$. We have $\gamma \in (1, 2]$.  When $\lambda $ goes to $\infty$, we get 
$$ \lim_{\lambda \rightarrow \infty} \lambda^{\gamma} 
\bE \left[ e^{-\lambda Z_\gamma } \right] = \frac{2^{\frac{\gamma}{\gamma -1} } }{ (\gamma -1)^{\frac{\gamma}{\gamma -1} }  } \; .$$ 
Then, $(ii)$ is a consequence of a Tauberian theorem due to Feller (\cite{Fe} Chapter XIII, § 5 ; see also  \cite{BiGoTe} Theorem 1.7.1', p. 38). \qed
\end{proof}
Recall the definition of $M^*_r (a)$ from (\ref{Mstardef}).  For any $0 \leq r^\prime \leq r \leq a $, we set 
\begin{equation}
\label{massshelldef}
 Q_{r^\prime , r} (a)=\sum_{j \in \cI^*} \un_{[a-r, a-r^\prime \, )} (r^*_j) \int_0^{\zeta^*_j}  \un_{\{ H^{*j}_s \leq r-a+ r^*_j \}} \; .
 \end{equation} 
Arguing as in Lemma \ref{couronne}, we prove the following independence property. 
\begin{lemma}
\label{shellific} Let $(r_n , n \geq 0)$ be a sequence such that $0 < r_{n+1} \leq r_n \leq a$ and $\lim_n r_n = 0$. Then, the random variables $(Q_{r_{n+1}, r_n } (a), n \geq 0)$ are independent. 
\end{lemma}
\begin{remark}
\label{massutile} Note that the increments of $r \in [0, a] \mapsto M^*_r (a)$ are not independent. However, for any $0 \leq r^\prime \leq r \leq a $, we have 
$$ M_r^* (a) -M^*_{r^\prime} (a)= Q_{r^\prime , r} (a)+ \sum_{^{j \in \cI^*}} \un_{[ a-r^\prime \, , \, a ]} (r^*_j) \int_0^{\zeta^*_j}  \un_{\{  r^\prime-a+ r^*_j < H^{*j}_s \leq r-a+ r^*_j \}} \; , $$
which first implies that $ M_r^* (a)  \geq Q_{r^\prime , r} (a)$. Moreover, we easily see that $Q_{r^\prime , r} (a)$ is independent of $M^*_{r^\prime} (a)$.  \cq 
\end{remark}
Recall the definition of $\kappa_a (\lambda, \mu)$ and recall it satisfies (\ref{equakappa}). 
\begin{lemma}
\label{MQstar} For any $a \in (0, \infty)$ and for any $r \in [0, a]$, we have 
\begin{equation}
\label{QLapcour}
\bE \big[ \exp (-\lambda  Q_{r^\prime , r} (a) ) \big]=  1 - \frac{\kappa_{r-r^\prime} (\lambda, 0)^\gamma}{\lambda} \; .
\end{equation}
Then, 
$$ (r-r^\prime)^{-\frac{\gamma}{\gamma-1}}   Q_{r^\prime , r} (a)  \overset{{\rm (law)}}{=}   
M^*_1 (1)   \; .$$ 
\end{lemma}
\begin{proof}
Recall that conditionally given $U$, $\cN^*$ is distributed as a Poisson point process with intensity $dU_t \otimes N_\gamma$. Thus,
$$ \bE \left[ \exp \big(- \lambda Q_{r^\prime , r} (a)  \big) \, | \, U \, \right]= \exp \Big(- \int_{[a-r, a-r^\prime\, )} \!\!\!\!\!\!\! \!\!\!\!\!\!\!  dU_t \,  \kappa_{r-a+t} (\lambda , 0) \Big) . $$ 
Since $U$ is a subordinator with Laplace exponent $\lambda \mapsto \gamma \lambda^{\gamma -1}$, we get 
$$ \bE \big[ e^{- \! \lambda  Q_{r^\prime , r} (a) } \big]= \exp \Big(\! \!-\! \gamma \!\!  
\int_{a-r}^{ a-r^\prime \, } \!\!\! \!\!\!\!\!\!\!  \kappa_{r-a+t}( \lambda, 0)^{\gamma -1} \!dt   \Big)=\exp \Big(\!\!-\! \gamma  \!\! 
\int_{0}^{ r-r^\prime \, }  \!\!\!  \!\!\!\!\!\!\!  \kappa_{s}( \lambda, 0)^{\gamma -1}ds   \Big)  . $$ 
Set $y= \kappa_s (\lambda, 0)$. Then, (\ref{equakappa}) entails 
$$ \gamma \int_{0}^{ r-r^\prime \, }  \!\!\!  \!\!\!\!\!\!\!  \kappa_{s}( \lambda, 0)^{\gamma -1}ds = \!\!\! \int_{0}^{\kappa_{r-r^\prime } (0, \lambda)} \, \frac{{\gamma\,  y^{\gamma-1} }}{{\lambda -y^\gamma }} \, dy= \log \lambda - \log \big( \lambda -\kappa_{r-r^\prime} (\lambda , 0)^\gamma \, \big)  ,  $$
which entails the first point of the lemma. Now observe that the scaling property (\ref{Hscaling}) combined with (\ref{kappaexc}) entails that for any $a, \lambda \geq 0$, and any $c >0$, one has 
$$  \kappa_{a} (\lambda, 0) = c^{\frac{1}{\gamma -1}} \kappa_{c\, a} (c^{-\frac{\gamma}{\gamma-1}} \lambda , 0 ) \; , $$ 
which easily implies the second point of the lemma. \qed 
\end{proof}
Take $r^\prime= 0$ in the previous lemma to get the following lemma. 
\begin{lemma}
\label{Mstar} For any $a \in (0, \infty)$ and for any $r \in [0, a]$, we have 
$$ \bE \big[ \exp (-\lambda M^*_r (a) ) \big]= 1 - \frac{\kappa_{r} (\lambda, 0)^\gamma}{\lambda} 
\; .$$
Then, $ r^{_{-\gamma/ (\gamma-1)}}M^*_r (a)$ has the same law as $M^*_1 (1) $. 
\end{lemma}
To simplify notation, let us set $Y_\gamma := M^*_{ 1} (1) $. 
\begin{proposition}
\label{massup}  For any $\gamma \in (1, 2)$, we have 
$$ \lim_{x \rightarrow \infty} x^{\gamma -1} \bP (Y_\gamma \geq x ) =  
\frac{1}{\Gamma (2-\gamma) } .$$
\end{proposition}
\begin{proof}
Recall (\ref{kappaexc}), recall that $\int_{^0}^{_\zeta} \un_{\{ H_s \leq a \}} ds =
 \int_{^0}^{_a} L^{_b}_{^\zeta} \, db$ and recall that $N(L^{_b}_{^\zeta} )= 1$, for any $b \in (0, \infty)$ (see (\ref{meanloc})). Thus, 
\begin{equation}
\label{intoccbelowa}
 \lim_{\lambda \rightarrow 0} \frac{\kappa_a (\lambda, 0)}{\lambda}=N \Big(  \int_0^\zeta \un_{\{ H_s \leq a \}} ds\Big)=  \int_0^a N \big( L^b_\zeta  \big) \, db = a \; .
\end{equation}
Take $a=1$ in (\ref{intoccbelowa}) and use Lemma \ref{Mstar} to get 
$$ \bE \left[ e^{-\lambda Y_\gamma }\right] = 1 - \lambda^{\gamma -1} + o (\lambda^{\gamma -1} )$$
when $\lambda$ goes to $0$. Since $0 < \gamma- 1< 1$, a 
Tauberian theorem due to Bingham and Doney \cite{BiDo} entails the desired result 
(see also \cite{BiGoTe} Theorem 8.1.6, p. 333). \qed
\end{proof}

\section{Proofs of the main results.}   
\label{Pfsec}   
\subsection{Proof of Theorem \ref{packlevel}.}
\label{proofpaclevel}   
Let us fix $a \in (0, \infty)$ and let $g: (0, 1) \rightarrow (0, \infty)$ be such that $\lim_{0+} r^{_{-1/ (\gamma-1)}} g(r)= 0$. To simplify notation we set $h(r)=  r^{_{-1/ (\gamma-1)}}g(r)$. Lemma \ref{Lstar} and Proposition \ref{levelestimates} $(ii)$ imply 
that for all sufficiently large $n$, 
\begin{equation}
\label{pacballevest}
 \bP ( L^*_{ 2^{-n}}  (a)   \leq g (2^{-n})  )=  \bP( Z_\gamma \leq h(2^{-n})) 
\sim_{n \rightarrow \infty} K_\gamma  h(2^{-n})^{\gamma } \; , 
\end{equation}
where $K_\gamma$ is the limit on the right member of Proposition \ref{levelestimates} $(ii)$. 
We first prove Theorem \ref{packlevel} $(i)$. So we assume 
\begin{equation}
\label{finih}
\sum_{n \geq 1} h (2^{-n})^{\gamma}  < \infty \; .
\end{equation}
Borel-Cantelli and (\ref{pacballevest}) imply $\bP ( \liminf_{n \rightarrow \infty} L^*_{ 2^{-n}}  (a) /g(2^{-n}) \geq  1)= 1$.  This easily entails $\bP ( \liminf_{n \rightarrow \infty} L^*_{ 2^{-n}}  (a) /g(2^{-n}) = \infty)= 1$, since (\ref{finih}) is also satisfied by $K.h$ for arbitrarily large $K$. Then, (\ref{keyloc}) implies 
$$ N_\gamma \left( \int_{\cT} \!\!\! \ell^a(d\sigma) \un_{ \{ \liminf_{n}  \ell^a (B(\sigma, 2^{-n})) / g(2^{-n})   < \infty      \}} \,   \right)=  0 \; , $$
which entails (\ref{lowdensinfty}) in Theorem \ref{packlevel} $(i)$. 
\begin{lemma}
\label{packlevbad} We assume that $g$ is a regular gauge function that satisfies (\ref{finih}) and we set 
$$ E= \big\{  \sigma \in \cT (a) \; : \; \liminf_{r \rightarrow 0}   \ell^a( B(\sigma ,r) ) / g(r) < 1 \big\} \; .$$ 
Then, $N_\gamma$-a.e.$\, \cP_g ( E ) \leq \langle \ell^a  \rangle $. 
\end{lemma}
\begin{proof} Let us fix $b\in (0, a)$. Recall that $(g^{_b}_{^j}, d^{_b}_{^j})$, $j\in \cI_b$
stand for the connected components of the open set $\{t \geq 0:H_t>b\}$ and recall that 
$H^{_{b, j}}$ is the corresponding excursion of $H$ above $b$ corresponding to $(g^{_b}_{^j}, d^{_b}_{^j})$. We set $\cT^{_{b}}_{^j}= p([g^{_b}_{^j}, d^{_b}_{^j}])$ and $\sigma^{_b}_{^j} = p(g^{_b}_{^j})= p(d^{_b}_{^j})$. As already mentioned $(\cT^{_b}_{^j}, d, \sigma^{_b}_{^j})$ is isometric to the tree coded by $H^{_{^{b, j}}}_{ }$. The total height of $\cT^{_b}_{^j}$ is then $\Gamma (\cT^{_b}_{^j})= \sup_{^{s \geq 0}}H^{_{b, j}}_s$. For any $\eta >0$, we set 
$$ \cD_{b, \eta}= \{ \cT^b_j \; ; \; i \in \cI_b \, : \; \Gamma (\cT^{_b}_{^j}) >\eta   \} \; .$$
Note that $\cD_{b, \eta}$ is a finite set. Observe that $\ell^a (\cT^{_b}_{^j} )= L^{_a}_{^{d^b_j}}-L^{_a}_{^{g^b_j}}$ is the local time at level $a-b$ of $H^{_{^{b, j}}}_{ }$, or equivalently the total mass of the local time measure at level $a-b$ of $\cT^{_b}_{^j}$. Then, the branching property entails for any $x>0$, 
$$ N_\gamma^{_{(b)}} \Big( \sum_{^{T \in  \cD_{b, a-b}}} \!\!\!\!\! \un_{ \{ \ell^a (T ) \leq x \}}
\; \Big| \,  \cG_b  \Big)  = L^b_\zeta   \; N_\gamma  \big(  \, L^{a-b}_\zeta  \, \leq x \, ; \, \sup H > a-b   \big) \; .$$
Recall that $L^{_{a-b}}_{^\zeta} = \langle \ell^{a-b} \rangle$. Then, (\ref{vvvequa}) and the scaling property (\ref{scalingZolo}) imply 
$$ N_\gamma  \big(  \, L^{a-b}_\zeta  \, \leq x \, ; \, \sup H > a-b   \big)= \big( (\gamma\!-\!1) (a\!-\!b) \big)^{-\frac{1}{\gamma-1}}  N_\gamma^{_{(1)}} \!\!  \big(  \; \langle \ell^{1} \rangle  \, \leq (a\!-\!b)^{-\frac{1}{\gamma-1}}x  \,   \big) .$$ 
Recall that $ N_\gamma$-a.e.$\, L^b_\zeta=\langle \ell_b \rangle =0$, on $\{ \sup H \leq b \}$. Thus, (\ref{meanloc}) and (\ref{vvvequa}) entail
\begin{equation}
\label{tricomput}
N_\gamma \Big( \sum_{^{T \in  \cD_{b, a-b}}} \!\!\!\!\! \un_{ \{ \ell^a (T ) \leq x \}}
 \Big) = \big( (\gamma\!-\!1) (a\!-\!b) \big)^{-\frac{1}{\gamma-1}}  N_\gamma^{_{(1)}}   \big( \;  \langle \ell^{1} \rangle  \, \leq (a\!-\!b)^{-\frac{1}{\gamma-1}}x  \,   \big). 
\end{equation}
For any $n \in \bN$ such that $2^{-n} <a$, we next set 
$$ V_n =  \sum_{ T \in \cD_{a-2^{-n} ,2^{-n} } }  g( 2. 2^{-n}) \un_{ \{ \ell^a (T ) \leq   g(2.2^{-n} )   \}}  \; . $$
We apply (\ref{tricomput}) with $b= a-2^{-n}$ and $\eta = 2^{-n}$, and we use Lemma \ref{taillevdirect} to get 
$$ N( V_n) =(\gamma-1)^{-\frac{1}{\gamma -1}} 2^{\frac{n}{\gamma -1}} g( 2. 2^{-n})  N^{_{_{(1)}}}_\gamma \!\! \big(\langle \ell^{1} \rangle \leq 2^{-\frac{1}{\gamma-1}} h(2.2^{-n} )  \big) \leq K^\prime_\gamma \, h(2.2^{-n} )^\gamma, $$
where $K^\prime_\gamma $ is a positive constant that only depends on $\gamma$. Therefore, (\ref{finih}) entails
\begin{equation}
\label{limidirect}
N_\gamma \;^{_{_-}}{\rm a.e.} \qquad \lim_{n \rightarrow \infty} \;  \sum_{p \geq n} V_p = 0 \; .
\end{equation}

 Let $\varepsilon \in (0, a/2)$. We assume that $\cT (a)  \neq \emptyset$. Let $(\bar{B}(\sigma_m, r_m)\, ; \, m \geq 1)$ be any $\varepsilon$-closed packing of $E$. Namely, the closed balls $\bar{B}(\sigma_m, r_m)$ are pairwise disjoints, $\sigma_m \in E \subset  \cT(a)$ and $r_m \leq \varepsilon$, for any $m \geq 1$. Let us fix $m \geq 1$. There exists $n$ (that depends on $m$) such that $2^{-n}< r_m \leq 2. 2^{-n}$. Now observe that $ \cT(a)$ is the union of the sets $T \cap \cT(a)$ where $T$ ranges in $\cD_{a-2^{-n-1} ,2^{-n-1} }$. Consequently, there exists $T^*\in \cD_{a-2^{-n-1} ,2^{-n-1} }$ such that $\sigma_m \in T^* \cap \cT(a)$. Denote by $\sigma^*$, the lowest point in $T^*$. Namely $\sigma^*$ is the point of $T^*$ that is the closest to root and $\sigma^* \in \cT(a-2^{-n-1})$. 
It is easy to prove that for any $\sigma \in T^*\cap \cT(a)$, we have 
$$d(\sigma, \sigma_m)\leq d(\sigma, \sigma^*) + d(\sigma, \sigma^*)= 2 . 2^{-n-1}= 2^{-n} < r_m \; .$$
 Thus $T^*\cap \cT(a) \subset \cT(a) \cap \bar{B}(\sigma_m, r_m)$. Thus, $\ell^a(T^*) \leq \ell^a(\bar{B}(\sigma_m, r_m))$. Since this holds true for any $m \geq 1$, we get  
\begin{equation}
\label{discreti}
\sum_{^{m \geq 1}} g(r_m)\un_{ \{  \ell^a(B(\sigma_m , r_m) ) < g(r_m) \}} \leq \sum_{n: 2^{-n}< \varepsilon} V_{n+1} .
\end{equation}
Now observe that 
$$ \sum_{^{m \geq 1}} g(r_m)\un_{ \{  \ell^a(B(\sigma_m , r_m) ) \geq g(r_m) \}} \leq  \sum_{^{m \geq 1}} \ell^a(B(\sigma_m , r_m) ) \leq \langle \ell^a \rangle .$$
This inequality combined with (\ref{discreti}) implies 
$$\sum_{m \geq 1} g(r_m) \leq \langle \ell^a \rangle + \sum_{ n: 2^{-n} <\varepsilon } V_{n+1} \; .$$
Since, this holds true for any $\varepsilon$-closed packing $(\bar{B}(\sigma_m, r_m)\, ; \, m \geq 1)$ of $E$, (\ref{limidirect}) entails $\cP^*_g (E) \leq \langle \ell^a \rangle $, $N_\gamma$-a.e.$\;$where $\cP^*_g$ stands for the $g$-packing pre-measure, which completes the proof of the lemma since $\cP_g (E) \leq \cP^*_g(E)$, by definition of $\cP_g$. \qed 
\end{proof}
Lemma \ref{genedens} $(ii)$ implies that $\cP_g (\cT (a) \backslash E) \leq  \langle \ell^a \rangle$. This inequality combined with Lemma \ref{packlevbad} entails $N_\gamma$-a.e.$\, \cP_g ( \cT (a) ) = \cP_g (E) + \cP_g (  \cT (a ) \backslash E ) \leq  2 \langle \ell^a \rangle $. This proves that for any regular gauge function $g$ that satisfies (\ref{finih}), we $N_\gamma$-a.e.$\;$have $\cP_g (\cT (a)) \leq  2 \langle \ell^a \rangle$. Thus, for any constant $K >0$, we have $\cP_{Kg}  (\cT (a)) \leq  2 \langle \ell^a \rangle$. Now observe that $\cP_{Kg}  (\cT (a))= K \cP_{g}  (\cT (a))$, which easily implies $\cP_g (\cT (a))= 0$, $N_\gamma$-a.e. This completes the proof of Theorem \ref{packlevel} $(i)$.

\medskip

Let us prove Theorem \ref{packlevel} $(ii)$. We now assume 
\begin{equation}
\label{infinih}
\sum_{n \geq 1} h (2^{-n})^{\gamma}  = \infty \; .
\end{equation}
For any $n \geq 1$, set $S_n = \varepsilon_1 + \ldots + \varepsilon_n$, where $ \varepsilon_n = \un_{\{  L^*_{ 2^{-n} }   (a) \;  \leq \; g (2^{-n})   \}} $. Estimates (\ref{pacballevest}) implies that 
\begin{equation}
\label{meansn}
\bE \left[  S_n \right] \sim_{n \rightarrow \infty} K_\gamma \sum_{l=1}^n h \left( 2^{-l} \right)^{\gamma}  \; .
\end{equation}
 Next observe that 
$$ \bE \left[ S_n^2 \right] = \bE[ S_n] + 2 \!\!\! \!\!\sum_{1 \leq k < l \leq n} \bE \left[ \varepsilon_k 
\varepsilon_l\right] \; .$$
Fix $1 \leq k < l \leq n$. As noted in Remark \ref{utile}, $\Lambda_{2^{-l} , 2^{-k} }   (a) \leq L^*_{ 2^{-k} }   (a) $. Thus, 
\begin{eqnarray*}
\{   L^*_{ 2^{-k} }   (a)   \leq  g (2^{-k})  \} \cap  \{  L^*_{ 2^{-l} }   (a) \;  \leq \; g (2^{-l})   \}   \hspace{30mm} \\ 
\hspace{20mm} \subset  \{   
\Lambda_{2^{-l} , 2^{-k} }   (a)   \leq  g (2^{-k})  \} \cap  \{  L^*_{ 2^{-l} }   (a) \;  \leq \; g (2^{-l})   \} .
\end{eqnarray*}
Remark \ref{utile} also asserts that  $L^*_{ 2^{-l} }   (a)$ is independent of $\Lambda_{2^{-l} , 2^{-k} }   (a)$. Thus, 
\begin{eqnarray}
 \label{borneun}
\bE \left[ \varepsilon_k 
\varepsilon_l \right] & \leq & \bP \left( \;  \Lambda_{2^{-l} , 2^{-k} }   (a)  \leq \; g (2^{-k}) \; \right)  \bE \left[ \varepsilon_l \right]  \nonumber \\
& \leq & \bP \left( \; 2^{\frac{k}{\gamma-1}} \Lambda_{2^{-l} , 2^{-k} }   (a)  \leq  h (2^{-k}) \; \right)  \bE \left[ \varepsilon_l \right] \; .
\end{eqnarray}
We give an upper bound of the last probability thanks to the Laplace transform of $2^{\frac{k}{\gamma-1}}\Lambda_{2^{-l} , 2^{-k} }   (a) $ that is explicitly given in (\ref{Lapcour}) in Lemma \ref{Lambstar}:
\begin{eqnarray*}
\bE \big[  \exp (-\lambda 2^{\frac{k}{\gamma-1}}\Lambda_{2^{-l} , 2^{-k} }   (a) ) \big]  &=& \left(  \frac{\frac{(\gamma \!-\! 1)}{2} 2^{-(l-k)} \lambda^{\gamma -1}+ 1}{ \frac{(\gamma \!-\! 1)}{2} \lambda^{\gamma -1}+ 1}  \right)^{\frac{\gamma }{\gamma -1}} \\
 &= & \Big( 2^{-(l-k)} +  \frac{ 1- 2^{-(l-k)} }{\frac{\gamma-1}{2}  \lambda^{\gamma -1}+ 1} 
 \Big)^{\frac{\gamma }{\gamma -1}} \\
& \leq & 2^{\frac{1}{\gamma-1}} \Big( 2^{-\frac{\gamma }{\gamma -1}(l-k)} + \Big(\frac{ 1- 2^{-(l-k)} }{\frac{\gamma-1}{2}  \lambda^{\gamma -1}+ 1}  \Big)^{\frac{\gamma }{\gamma -1}}    \Big), 
\end{eqnarray*}
by an elementary convex inequality. Set $C_1= 2^{_{1/ (\gamma-1)}} \left(2/(\gamma -1)\right)^{_{\gamma /(\gamma -1)}}$. The previous inequality easily entails the following 
$$ \bE \big[  \exp (-\lambda 2^{\frac{k}{\gamma-1}} \Lambda_{2^{-l} , 2^{-k} }   (a) ) \big]   \leq   C_1 \left( 2^{-\frac{\gamma }{\gamma -1}(l-k)} + \lambda^{-\gamma} \right)  \; .$$
We now use Markov inequality to get 
\begin{eqnarray}
\label{bornedeuz}
\bP \left( \;  2^{\frac{k}{\gamma-1}}\Lambda_{2^{-l} , 2^{-k} }   (a)  \leq \; h (2^{-k}) \; \right)  & \leq & e
\bE \Big[ \exp \big( -2^{\frac{k}{\gamma-1}} \Lambda_{2^{-l} , 2^{-k} }   (a) / h(2^{-k}) \big)  \Big] \nonumber \\
& \leq & e \, C_1 \big( 2^{-\frac{\gamma }{\gamma -1}(l-k)}  + h (2^{-k} )^\gamma  \big) \; . 
\end{eqnarray}
Now, (\ref{pacballevest}) implies that there exists $C_2\in (0, \infty)$ that only depends on $\gamma$ and $h$ such that $ h( 2^{-k} )^{_\gamma} \leq C_2 \bE [ \varepsilon_k ]  $, for any $k \geq 1$. Thus, (\ref{borneun}) and (\ref{bornedeuz}) imply there exists $C_3 \in (0, \infty)$ (that only depends on $h$ and $\gamma$) such that 
$$ \bE \left[ \varepsilon_k 
\varepsilon_l \right]  \leq C_3  \left( 2^{-\frac{\gamma }{\gamma -1}(l-k)}  \bE [ \varepsilon_l ]   + \bE [ \varepsilon_k ] \bE [ \varepsilon_l ] \right) \; , $$ 
which easily implies 
$$ \bE \left[ S_n^2 \right] \leq   \left(1 + \frac{C_3}{1-2^{-\frac{\gamma }{\gamma -1}}} \right)
 \bE \left[ S_n \right] + C_3 . \left(  \bE \left[ S_n\right] \right)^2 \; .$$
By (\ref{meansn}) and (\ref{infinih}),we get 
$$ \limsup_{n \rightarrow \infty} \frac{\bE \left[ S_n^2 \right]}{\left(  \bE \left[ S_n\right] \right)^2} \leq C_3 $$
and Kochen-Stone's Lemma implies $ \bP ( \sum_{^{n \geq 1}} \varepsilon_n = \infty ) \geq 1/C_3 >0 $. As observed in Remark \ref{utile}, $r\rightarrow L^*_r (a)$ has independent increments. Thus, Kolmogorov's $0$-$1$ law applies and we get $ \bP( \sum_{^{n \geq 1}} \varepsilon_n = \infty)= 1 $. This entails 
$\bP ( \liminf_n L^*_{2^{-n}} (a)/g(2^{-n}) \leq 1 ) = 1$. 
Observe that (\ref{infinih}) is also satisfied by $c.h$ for arbitrarily small $c>0$. This easily implies $\bP ( \liminf_n L^*_{2^{-n}} (a)/g(2^{-n}) =0 ) = 1$ and (\ref{keyloc}) entails 
$$ N_\gamma \left( \int_{\cT} \!\!\! \ell^a(d\sigma) \un_{ \{ \liminf_{n}  \ell^a (B(\sigma, 2^{-n})) / g(2^{-n})   >0       \}} \,   \right)=  0 \; . $$
This proves (\ref{lowdensnull}) in Theorem \ref{packlevel} $(ii)$. Furthermore, if $g$ is a regular gauge function, then, (\ref{lowdensnull}) and Lemma \ref{genedens} $(i)$ entail that $N_\gamma$-a.e.$\, \cP_g (\cT (a) \, )= \infty$, on $\{ \cT (a) \neq \emptyset \}$, which completes the proof of Theorem \ref{packlevel}. \qed 

\subsection{Proof of Proposition \ref{Htestlv}.}
\label{proofHtestlv}
Fix $a >0$ and let $g$ be as in Proposition \ref{Htestlv}. Namely $g: (0,1) \rightarrow (0, \infty)$ is such that $\lim_{0+} r^{_{-1/(\gamma -1)}} g(r)= \infty$. To simplify notation we set $h(r)= r^{_{-1/(\gamma -1)}} g(r)$. Although Proposition \ref{Htestlv} $(i)$ is already proved in \cite{DuLG3} we provide a brief proof of it:  We assume that 
\begin{equation}
\label{fifinini}
 \sum_{n \geq 1} h (2^{-n})^{-(\gamma-1)}  < \infty \; .
 \end{equation}
The scaling property stated in Lemma \ref{Lstar} and Proposition \ref{levelestimates} $(i)$ imply 
that for all sufficiently large $n$, 
$$ \bP ( L^*_{2^{-n}} (a) \geq g (2^{-n})  )=  \bP( Z_\gamma \geq h(2^{-n})) \sim_{n \rightarrow \infty} \frac{\gamma}{2\Gamma (2-\gamma)} h(2^{-n})^{-(\gamma -1)} \; .$$ 
Borel-Cantelli entails $\bP( \limsup_{n \rightarrow \infty} L^*_{2^{-n}} (a) /g(2^{-n}) \leq  1 )$. Since (\ref{fifinini}) is satisfied by $K.h$ for arbitrarily large $K$, we easily get  $\bP( \limsup_{n \rightarrow \infty} L^*_{2^{-n}} (a) /g(2^{-n}) =0 )= 1$ and (\ref{keyloc}) entails 
$$ N_\gamma \left( \int_{\cT} \!\!\! \ell^a(d\sigma) \un_{ \{ \limsup_{n}  \ell^a (B(\sigma, 2^{-n})) / g(2^{-n})   >0       \}} \,   \right)=  0 \; . $$
This proves (\ref{uppdensnull}) in Proposition \ref{Htestlv} $(i)$. Furthermore, if $g$ is a regular gauge function, then, (\ref{uppdensnull}) and Lemma \ref{genedens} $(iii)$ entail that $N_\gamma$-a.e.$\, \cH_g (\cT (a) \, )= \infty$, on $\{ \cT (a) \neq \emptyset \}$, which completes the proof of Proposition \ref{Htestlv} $(i)$. 

\medskip

Let us prove (\ref{uppdensinfty}) in Proposition \ref{Htestlv} $(ii)$. We now assume 
\begin{equation}
\label{ininifif}
\sum_{n \geq 1} h (2^{-n})^{-(\gamma-1)}  = \infty \; .
\end{equation}
For any $n \geq 1$, set  $ \varepsilon_n = \un_{\{ \Lambda_{2^{-n-1}, 2^{-n}} (a ) \;  \geq \; g (2^{-n})   \}} $. The scaling property in Lemma \ref{Lambstar} and Proposition \ref{levelestimates} $(i)$ imply 
$$ \bE [\varepsilon_n ] \sim_{n \rightarrow \infty}  \frac{\gamma}{4\Gamma (2-\gamma)} h(2^{-n})^{-(\gamma -1)} \; .$$ 
Therefore $ \sum_{n \geq 1} \bE [\varepsilon_n ] = \infty$. The independence property stated in Lemma 
\ref{couronne} shows that the $\varepsilon_n$'s are independent. The converse of Borel-Cantelli  implies $ \bP (\sum_{^{n \geq 1}} \varepsilon_n = \infty)= 1$. 
As noticed in Remark \ref{utile}, we have $ \varepsilon_n \leq \un_{\{ L^*_{2^{-n}} (a) \geq g (2^{-n}) \} } $.  
Consequently, $\bP( \limsup_{n \rightarrow \infty } L^*_{2^{-n}} (a)/  g (2^{-n}) \geq 1)= 1$. 
Since (\ref{ininifif}) is satisfies by $c.h$ for arbitrarily small $c>0$, we easily get $\bP( \limsup_{n \rightarrow \infty } L^*_{2^{-n}} (a)/  g (2^{-n})= \infty)= 1$ and (\ref{keyloc}) entails 
$$ N_\gamma \left( \int_{\cT} \!\!\! \ell^a(d\sigma) \un_{ \{ \limsup_{n}  \ell^a (B(\sigma, 2^{-n})) / g(2^{-n})   < \infty      \}} \,   \right)=  0 \; ,$$
which proves (\ref{uppdensinfty}) in Proposition \ref{Htestlv} $(ii)$. \qed 

\subsection{Proof of Proposition \ref{Htestmass}.}
\label{proofHtestmass}
Let $g$ be as in Proposition \ref{Htestmass}. Namely $\lim_{0+} r^{_{-\gamma/(\gamma -1)}} g(r)= \infty$. To simplify notation we set $h(r)= r^{_{-\gamma/(\gamma -1)}} g(r)$. Although Proposition \ref{Htestmass} $(i)$ is already proved in \cite{DuLG3} we provide a brief proof of it. We assume that 
\begin{equation}
\label{cestfini}
 \sum_{n \geq 1} h (2^{-n})^{-(\gamma-1)}  < \infty \; .
 \end{equation}
Let us fix $a >0$. The scaling property stated in Lemma \ref{Mstar} and Proposition \ref{massup} $(i)$  imply that 
$$ \bP (  M^*_{2^{-n}} (a) \geq g (2^{-n})  )=  \bP( Y_\gamma \geq h(2^{-n})) \sim_{n \rightarrow \infty} \frac{h(2^{-n})^{-(\gamma -1)} }{\Gamma (2-\gamma)}\; .$$ 
Borel-Cantelli implies $\bP( \limsup_{n} M^*_{2^{-n}} (a) /  g (2^{-n})  \!\leq \! 1) \!= \!1 $. Since (\ref{cestfini}) is satisfied by $K.h$ for arbitrarily large $K$, we easily get $\bP( \limsup_{n \rightarrow \infty} M^*_{2^{-n}} (a) /  g (2^{-n})  = 0) = 1 $. By (\ref{keymass}), for any $a>0$, we get 
$$ N_\gamma \left( \int_{\cT} \!\!\! \ell^a(d\sigma) \un_{ \{ \limsup_{n}  \bm (B(\sigma, 2^{-n})) / g(2^{-n}) >0 \}} \,   \right)=  0 \; .$$
Since $\bm = \int_0^\infty \ell^a$, this entails (\ref{massuppdensnull}) in Proposition \ref{Htestmass}. Furthermore, if $g$ is a regular gauge function, then, (\ref{massuppdensnull}) and Lemma \ref{genedens} $(iii)$ imply that $\cH_g (\cT  )= \infty$, $N_\gamma$-a.e.$\;$, which completes the proof of Proposition \ref{Htestmass} $(i)$.

\medskip

  Let us prove (\ref{massuppdensinfty}) in Proposition \ref{Htestmass} $(ii)$. We assume 
\begin{equation}
\label{cestpasfini}
\sum_{n \geq 1} h (2^{-n})^{-(\gamma-1)}  = \infty \; .
 \end{equation}
For any $n \geq 1$, we set $ \varepsilon_n = \un_{\{ Q_{2^{-n-1}, 2^{-n}} (a)  \;  \geq \; g (2^{-n})   \}} $. The scaling property stated in Lemma \ref{MQstar} and Proposition \ref{massup} $(i)$ entail 
$$ \bE [\varepsilon_n ] = \bP \big(  Y_\gamma \geq 2^{\frac{-\gamma }{\gamma -1}}   h(2^{-n}) \, \big) \sim_{n \rightarrow \infty}   \frac{2^{\gamma}h(2^{-n})^{-(\gamma -1)} }{\Gamma (2-\gamma)}\; , $$ 
Thus, $ \sum_{^{n \geq 1}} \bE [\varepsilon_n ] = \infty $. The independence property of Lemma \ref{shellific} $(i)$ implies that the $\varepsilon_n$'s are independent. Thus, $ \bP( \sum_{^{n \geq 1}} \varepsilon_n = \infty )= 1$, by the converse of Borel-Cantelli. Then, Remark \ref{massutile} entails $\varepsilon_n \leq \un_{\{ M^*_{ 2^{-n}}  (a)  \geq g (2^{-n}) \} }$, for any $n \geq 1$. 
Thus, $\bP( \limsup_{n \rightarrow \infty} M^*_{ 2^{-n}}  (a)  / g (2^{-n})  \geq  1)= 1 $. Since (\ref{cestpasfini}) is satisfied by $c.h$ for arbitrarily small $c>0$, we easily get $\bP( \limsup_{n \rightarrow \infty} M^*_{ 2^{-n}}  (a)  / g (2^{-n}) = \infty)= 1 $. By (\ref{keymass}), for any $a >0$, we get 
$$ N_\gamma \left( \int_{\cT} \!\!\! \ell^a(d\sigma) \un_{ \{ \limsup_{n}  \bm (B(\sigma, 2^{-n})) / g(2^{-n}) <\infty  \}} \,   \right)=  0 \; .$$
Since $\bm = \int_0^\infty \ell^a$, this entails (\ref{massuppdensinfty}) in Proposition \ref{Htestmass}.

\subsection{Proof of  Theorem \ref{noreguHauslv} and of Theorem \ref{noreguHausmass}.}
\label{proofnoHaus}

We fix $\gamma \in (1, 2)$ and we consider the $\gamma$-stable tree $(\cT, d, \rho)$ coded by the height process $(H_t, t \geq 0)$ under its excursion measure $N_\gamma$. Recall that $p$ stands for the canonical projection from $[0, \zeta]$ onto $\cT= [0, \zeta]/ \sim$. Recall that $\rho= p(0)$ stands for the root of $\cT$. We extend $p$ on $[0, \infty)$ by setting $p(t)= \rho $, for any $t \geq \zeta$. Let $0 \leq s \leq t$ and set $\cT_{s, t}= p([s, t])$, equipped with the distance $d$ on $\cT$. We set $\rho_{s,t}= p(r_0)$ where $r_0\in [s, t]$ is such that $H_{r_0}= \inf_{r\in [s, t]}H_r$. Observe that $(\cT_{s, t}, d, \rho_{s,t})$ is a compact rooted real tree that is isometric to the compact real tree coded by the process $H^{s, t}:=(H_{(s+r)\wedge t}, r \geq 0)$.

    Let $g: (0, 1) \rightarrow (0, \infty)$ be regular gauge function. We denote by $\cH_g$ the $g$-Hausdorff measure on $(\cT, d)$. Recall that $\cT$ is compact and note that any subset of $\cT$ is contained in a closed ball with the same diameter. In the definition of $\cH_g(\cT_{s,t})$, we may restrict our attention to finite coverings with closed balls with center of the form $p(r)$, with $r \in \bQ \cap [s, t]$ and with rational radius. This entails that $\cH_g (\cT_{s,t})$ is a measurable function of $H^{s,t}$. Similarly, for any $a \geq 0$, $\cH_g (\cT (a) \cap \cT_{s, t})$ is a measurable function of $H^{s,t}$. 

  Let us fix $a>0$. Recall the definition of $\tilde{H}^a$ that is the height process below $a$ and recall that $\cG_a$ is the sigma-field generated by $\tilde{H}^a$ augmented with the $N_\gamma$-negligible sets. We denote by $\cG_{a-}$ the sigma-field generated by $\bigcup_{^{b<a}} \cG_b$. It is easy to observe that $N_\gamma$-a.e.$\, \tilde{H}^a$ is the limit in $\bbD ([0, \infty), \bR)$ of $\tilde{H}^b$ when $b$ goes to $a$. Then, $\cG_{a-}= \cG_a$. Next observe that the rooted real tree coded by $\tilde{H}^a$ is isometric to $\bar{B} (\rho , a)= \{ \sigma \in \cT : d(\rho , \sigma) \leq a \}$. Thus, $\cH_g (\bar{B}(\rho, a)\, )$ and $\cH_g ( \cT( a) \, )$ are $\cG_a$-measurable $[0, \infty]$-valued random variables.

\subparagraph{Proof of Theorem \ref{noreguHauslv}.}

Let $g(r)= r^q s(r)$ where $q$ is nonnegative and where $s$ is slowly varying at $0$. Recall that we furthermore assume that $g$ is a regular gauge function. Recall that $N_\gamma$-a.e.$\;$on $\{ \cT(a) \neq \emptyset \}$, the Hausdorff dimension of $\cT(a)$ is $1/(\gamma-1)$. Thus, if $q > 1/(\gamma -1)$, then $\cH_g (\cT(a) \, )= 0$, $N_\gamma$-a.e. and if $q < 1/(\gamma -1)$, then $\cH_g (\cT(a) \, )= \infty$, $N_\gamma$-a.e.$\;$on $\{ \cT(a) \neq \emptyset \}$. 
{\it We then restrict our attention to the case $q= 1/(\gamma-1)$}.

\medskip

{\it The general idea of the proof of Theorem \ref{noreguHauslv}} is the following: if for a certain $a\in (0, \infty)$, we have $N_\gamma (\, 0< \cH_g ( \cT (a)) < \infty ) >0$, then we first prove that $ 0 <\cH_g ( \cT (a))< \infty $, $N_\gamma$-a.e.$\;$on $\{ \cT (a) \neq  \emptyset \}$. We next observe that $\cH_g (\cdot \cap \{ \cT(a) \neq  \emptyset \})$ behaves like $\ell^a$ with respect to the scaling property and the branching property and we prove it entails that $\cH_g (\cdot \cap \{ \cT(a) \neq  \emptyset \})= c_0\ell^a$, where $c_0\in (0, \infty)$. Finally, we get a contradiction thanks to the test stated in Proposition \ref{Htestlv}. 

\medskip

The proof is in several steps. We first discuss how $a \mapsto \cH_g ( \cT (a)) \in [0 , \infty]$ behaves with respect to the branching property. We agree on the convention $\exp (-\infty )= 0$. Then, for any $a, \lambda \in (0, \infty)$, it makes sense to set 
$$ \tilde{u} (a, \lambda)= N_\gamma \big(1-e^{-\lambda \cH_g (\cT(a) \, ) } \big) \; , $$
Recall from (\ref{vvvequa}) the definition of $v(a)$ and observe that 
$\tilde{u} (a, \lambda)\leq v(a) <\infty$. Let us fix $b \in (0, a )$. Recall that $(g_{^j}^{_b}, d_{^j}^{_b})$, $i \in \cI_b$ stands for the connected components of the open set $\{ t \geq 0: H_t >b \}$ and recall that for any $j \in \cI_b$, we denote by $H^{_{b, j}}$ the corresponding excursion of $H$ above $b$. We also set $\cT_{^j}^{_b}= p([g_{^j}^{_b}, d_{^j}^{_b}])$ and $\sigma_{^j}^{_b} = p(g_{^j}^{_b})$. Then, the subtree $(\cT_{^j}^{_b}, d, \sigma_{^j}^{_b})$ is isometric to the rooted compact real tree coded by the excursion $H^{_{b,j}}$. For any $n \geq 1$, we set $ L_n = \sum_{^{j \in \cI_b}} n \wedge \cH_g (\cT_{^j}^{_b} (a-b) \,) $, where $\cT_{^j}^{_b} (a-b):= \{ \sigma \in \cT_{^j}^{_b} : d(\sigma_{^j}^{_b} , \sigma)= a-b \}$. Note that $\cT_{^j}^{_b} (a-b)$ is the $(a-b)$-level set of $\cT_{^j}^{_b}$. Since $\cH_g (\cT_{^j}^{_b} (a-b) \,)$ is a measurable function of $H^{_{b,j}}$, the branching property (\ref{branchprop}) applies and for any $\lambda \in (0, \infty)$, we $N_\gamma^{_{_{(b)}}} $-a.s.$\;$get 
$$N_\gamma^{_{_{(b)}}} \big( e^{-\lambda L_n } \, \big| \, \cG_b \big)= e^{-L^b_\zeta  \tilde{u}_n (a-b, \lambda) } \; ,$$
where $ \tilde{u}_n (a-b, \lambda) = N_\gamma (1-\exp (-\lambda n \wedge \cH_g (\cT(a-b) \, ) ) \, )$. By monotone convergence, we get $\lim_n  \tilde{u}_n (a-b, \lambda) =  \tilde{u} (a-b, \lambda) $. Then, observe that  $\lim_n \uparrow L_n = \cH_g (\cT(a) \, )$. Thus, the conditional dominated convergence theorem implies that 
for any $\lambda \in (0, \infty)$, we $N_\gamma^{_{_{(b)}}} $-a.s.$\;$have  
\begin{equation}  
\label{branchHauslv}
N_\gamma^{_{_{(b)}}} \big( e^{-\lambda \cH_g (\cT(a) \, ) } \, \big| \, \cG_b \big)= e^{-L^b_\zeta  \tilde{u} (a-b, \lambda) } \; ,
\end{equation}
Since $L^{_b}_{^\zeta}= 0$, $N_\gamma$-a.e.$\;$on $\{ \sup H \leq b \}$, this entails 
\begin{equation}  
\label{Hausflot}
\tilde{u} (a, \lambda)=  N_\gamma \big(1-e^{-L^b_\zeta  \tilde{u} (a-b, \lambda) } \, \big)= u( b,\tilde{u} (a-b, \lambda) \, ) \; .
\end{equation}
Note that Theorem \ref{noreguHauslv} is implied by the two following claims. 
\begin{description}  
\item[{\bf (Claim 1)}] \hspace{-3mm} If there exists $a_0 \in (0, \infty)$ such that $N_\gamma (\cH_g (\cT(a_0) \, )    = \infty ) >0$, then for any $a \in (0, \infty)$, $N_\gamma$-a.e.$\, \cH_g (\cT(a) )  = \infty $, on $\{ \cT(a) \neq \emptyset \}  $.

\medskip

\item[{\bf (Claim 2)}]  \hspace{-3mm} For any $ a\in (0, \infty)$, $N_\gamma (0 < \cH_g (\cT (a)) < \infty) = 0$. 
\end{description}
We first prove (Claim 1). To that end, observe that $\tilde{u} (b, 0+)= \lim_{\lambda \rightarrow 0} \tilde{u} (b, \lambda)=N_\gamma (\cH (\cT(b)  ) = \infty)$ for any $b \in (0, \infty)$.  Then, (\ref{Hausflot}) entails 
\begin{equation}
\label{Hausflotnull}
\tilde{u} (a, 0+)= u( a-b,\tilde{u} (b, 0+) \, ) \, ,  \quad a>b >0 \; .
\end{equation}
Let us now recall the scaling property of $\cT$: Let  $c \in (0, \infty)$. The "law" of $(\cT, cd, \rho)$ under 
$N_\gamma$ is the "law" of $(\cT, d, \rho)$ under $c^{_{1/(\gamma-1)}}N_\gamma$. We next denote by $\cH_{g, cd}$ the $g$-Hausdorff measure on $(\cT, cd, \rho)$ and we set $g_c (r)= g(cr)$, for any $r \in (0, 1)$. Then, for any $b >0$, we easily get 
\begin{eqnarray*}
\cH_{g, c\, d} \big( \{ \sigma\in \cT : c\, d(\rho, \sigma)= b\}  \big) &= & \cH_{g_c} \big(  \{ \sigma\in \cT : d(\rho, \sigma)= b/c\}  \big) \\
&= & c^{\frac{1}{\gamma-1}}\cH_{g} \big(\cT (b/c) \,  \big), 
\end{eqnarray*}
since $g$ is regularly varying at $0$ with exponent $1/(\gamma-1)$. The scaling property for $\cT$ implies that the law of $c^{_{1/(\gamma-1)}}\cH_{g} (\cT (b/c) )$ under $N_\gamma$ is the same as the law of $\cH_{g} (\cT (b) )$ under $c^{_{1/(\gamma-1)}}N_\gamma$. Thus $\tilde{u} (b, 0+)= b^{_{-1/(\gamma-1)}}\tilde{u} (1, 0+)$, for any $b>0$. If there exists $a_0 >0$ such that $N_\gamma (\cH_g (\cT(a_0)  )    = \infty ) >0$, then $\tilde{u} (1, 0+) >0$, $\lim_{b \rightarrow 0} \tilde{u} (b, 0+)= \infty$. Recall that 
$$ v(a)= N_\gamma ( \sup H >a) = N_\gamma ( L^a_\zeta >0)= N_\gamma ( \cT (a) \neq \emptyset ) = \lim_{\mu \rightarrow \infty} u(a, \mu) \; , $$
where $u$ is given by (\ref{explisolbranch}). Then, (\ref{Hausflotnull}) and the previous 
arguments easily imply that for any $a \in (0, \infty)$, $\tilde{u} (a, 0+)= \lim_{b \rightarrow 0}  u( a-b,\tilde{u} (b, 0+) \, )= v(a)$. Namely, $N_\gamma (\cH (\cT(a)  ) = \infty) = N_\gamma (\cT( a) \neq \emptyset )$. Since $\{\cH (\cT(a)  ) = \infty \} \subset \{ \cT( a) \neq \emptyset \}$, it implies that 
$N_\gamma$-a.e.$\, \cH_g (\cT(a) \, )  = \infty $, on $\{ \cT(a) \neq \emptyset \} $, which proves the first claim.

\medskip

To prove (Claim 2), we argue by contradiction and we suppose that there exists $a_0 \in (0, \infty)$ such that 
\begin{equation}
\label{hypabsulv}
N_\gamma ( 0 < \cH_g (\cT(a_0) \, )< \infty) >0 \; .
\end{equation}
The previous arguments show that $\tilde{u} (b, 0+)=N_\gamma (\cH (\cT(b) \, ) = \infty) = 0$, for any $b \in (0, \infty)$. The scaling property discussed above entails 
\begin{equation}
\label{scaleHauslv}
\tilde{u} (b, \lambda)= c^{-\frac{1}{\gamma-1}} \tilde{u} \big(b/c \, , \,  c^{\frac{1}{\gamma-1}}\lambda \big) \; , \quad b, \lambda , c >0 \; .
\end{equation}
We first claim that for any $a \in (0, \infty)$, 
\begin{equation}
\label{posilv}
N_\gamma(\cH (\cT(a) \, )  >0) = N_\gamma (\sup H >a) = v(a) .
\end{equation}
Indeed, observe that $\tilde{v} (b):= \lim_{^{\lambda \rightarrow \infty}} \uparrow  \tilde{u} (b, \lambda)=N_\gamma (\cH (\cT(b) \, ) >0) $. Then, (\ref{scaleHauslv}) implies $\tilde{v} (b)= b^{_{-1/(\gamma-1)}}\tilde{v} (1)$. Assumption (\ref{hypabsulv}) entails that $ 0 < \tilde{u} (a_0, \lambda) \leq \tilde{v} (a_0)$. Thus, we get $\tilde{v} (1) >0$, which implies $\lim_{^{b \rightarrow 0}} \tilde{v} (b)= \infty$. Thanks to (\ref{Hausflot}), we get $\tilde{v} (a)= u( a-b,\tilde{v} (b) \, )$ and $\tilde{v} (a)= \lim_{^{b \rightarrow 0}}  u( a-b,\tilde{v} (b) \, )= v(a)$, which is (\ref{posilv}).

\medskip

 Recall that for any fixed $\lambda \in (0, \infty)$, $b \mapsto u(b, \lambda)$ is decreasing. Then, (\ref{Hausflot}) implies that $\tilde{u} (a, \lambda) \leq \tilde{u} (b, \lambda)$, for any $a >b >0$, and for any $\lambda >0$. Thus, it makes sense to set 
$\phi (\lambda)= \lim_{^{b\downarrow 0}} \uparrow \tilde{u} (b, \lambda) \in (0, \infty]$. Then (\ref{Hausflot}) entails $\tilde{u} (a, \lambda)= u(a, \phi(\lambda))$, for any $a, \lambda >0$, with the convention: $u(a, \infty)= v(a)$. Since $N_\gamma (\cH_g (\cT(a))= \infty) = 0$,  (\ref{posilv}) and the definition of $\tilde{u}$ imply $\tilde{u} (a, \lambda)< v(a)$. Consequently, $\phi (\lambda) \in (0,  \infty)$, for any $\lambda >0$. Next, observe that $u$ satisfies the same scaling property (\ref{scaleHauslv}) as $\tilde{u}$. Therefore,  $c^{_{1/(\gamma-1)}} \phi(\lambda)= \phi( c^{_{1/(\gamma-1)}} \lambda)$, for any $c, \lambda >0$. Namely, $\phi (\lambda)= c_0 \lambda$, where $c_0 := \phi (1) \in (0, \infty)$ and we have proved that 
\begin{equation}
\label{proporlv}
 \tilde{u} (b, \lambda)= u(b, c_0 \lambda) \; , \quad \lambda , b >0 \; .
 \end{equation}

We next prove that for any $a>b >0$, and for any $\lambda \geq 0$, 
\begin{equation}
\label{condilv}
N_\gamma^{_{_{(a)}}}  \,^{_{_-}}{\rm a.s.} \quad  N_\gamma^{_{_{(a)}}} \! \big( e^{-\lambda \cH_g (\cT(a) \, )}  \, \big| \, \cG_b \big)= N_\gamma^{_{_{(a)}}} \! \big( e^{-\lambda c_0 L^a_\zeta  }  \, \big| \, \cG_b \big)\; .
\end{equation}
{\it Proof of (\ref{condilv}):} by the branching property, we easily get  $N_\gamma^{_{_{(b)}}} \big( e^{-\lambda c_0L^a_\zeta  } \, \big| \, \cG_b \big)= e^{-L^b_\zeta  u (a-b, c_0\lambda) }$. Therefore, 
$$N_\gamma^{_{_{(b)}}} \big( e^{-\lambda \cH_g (\cT(a) \, ) } \, \big| \, \cG_b \big)= N_\gamma^{_{_{(b)}}} \big( e^{-\lambda c_0L^a_\zeta  } \, \big| \, \cG_b \big) . $$
Then, we get 
$N^{_{_{(b)}}}_\gamma (\un_{\{ \cH_g (\cT(a) \, ) = 0\}} | \cG_b)= N^{_{_{(b)}}}_\gamma (\un_{\{ L^a_\zeta= 0\}} | \cG_b)=e^{-L^b_\zeta  v (a-b) } $, by letting $\lambda $ go to $\infty$. Thus, $N_\gamma^{_{_{(b)}}}$-a.s. 
$$ N_\gamma^{_{_{(b)}}} \big( \un_{\{ \cH_g (\cT(a) \, ) >0\}} e^{-\lambda \cH_g (\cT(a) \, ) } \,  \big| \, \cG_b \big) = N_\gamma^{_{_{(b)}}} \big( \un_{\{ L^a_\zeta >0 \}} e^{-\lambda c_0L^a_\zeta  } \, \big| \, \cG_b \big)  
. $$
By (\ref{vvvequa}) and (\ref{posilv}), we have $\un_{\{ L^a_\zeta >0  \}}=\un_{\{ \cH_g (\cT(a) \, ) >0\}}=\un_{\{ \sup H >a\}}$, $N_\gamma$-a.e. Thus,  $N_\gamma^{_{_{(b)}}} $-a.s.$\; $we get 
$$N_\gamma^{_{_{(b)}}} \big( \un_{\{ \sup H >a \}} e^{-\lambda \cH_g (\cT(a) \, ) } \, \big| \, \cG_b \big)=
N_\gamma^{_{_{(b)}}} \big( \un_{\{ \sup H >a  \}} e^{-\lambda c_0L^a_\zeta  } \, \big| \, \cG_b \big) 
\; .$$
Recall that $N_\gamma^{_{_{(b)}}}\!= N_\gamma ( \cdot  \cap \{  \sup H \!>\!b\})/v(b)$ and note that $\{  \sup H \!> \!a\} \!\subset\! \{  \sup H \! >\!b\}$. Thus, for any positive $\cG_b$-measurable random variable $Y$, we get 
$$   N_\gamma \big( \un_{\{ \sup H >a \}} e^{-\lambda \cH_g (\cT(a) \, ) } Y  \big)
=N_\gamma\big( \un_{\{ \sup H >a \}} e^{-\lambda c_0L^a_\zeta } Y \big) \; , $$
which easily entails (\ref{condilv}). \qed

\medskip

  Recall that $L^{_a}_{^\zeta}$ and $\cH_g (\cT(a) \, )$ are $\cG_a$-measurable and recall that $\cG_{a-}= \cG_a$. By letting $b$ go to $a$ in (\ref{condilv}), we get $\cH_g (\cT(a) \, )=c_0 L^{_a}_{^\zeta}$, 
$N_\gamma^{_{_{(a)}}}$-a.s.$\; $which easily entails $\cH_g (\cT(a) \, )=c_0L^{_a}_{^\zeta} $, $N_\gamma$-a.e. Recall that $\ell^a(\cT(a) \, )= L^{_a}_{^\zeta}$. Thus, we have proved:
\begin{equation}
\label{coincidlv}
\forall a \in (0, \infty) \; , \quad N_\gamma \,^{_{_-}}{\rm a.e.} \quad \cH_g (\cT(a) \, )=c_0 \ell^a(\cT(a) \, ) \; .
\end{equation}

  We now prove the following. 
\begin{equation}
\label{equaloclv}
\quad N_\gamma \,^{_{_-}}{\rm a.e.} \quad \cH_g ( \, \cdot \, \cap \cT(a) \, )=c_0 \ell^a \; .
\end{equation}
{\it Proof of (\ref{equaloclv}):} let $a >b >0$. The branching property and (\ref{coincidlv}) easily imply that $N_\gamma$-a.e.$\;$for any $j \in \cI_{^b}$, we have $\cH_g (\cT_{^j}^{_b}(a-b) \, )=c_0 \ell^a(\cT_{^j}^{_b}(a-b) \, )$. For any $\sigma \in \cT(a)$ and any $r \geq 0$, we set $\bar{B}_a (\sigma , r)= \{ \sigma^\prime \in \cT(a): d( \sigma, \sigma^\prime) \leq r \}$. Note that $\bar{B}_a (\sigma , r)$ is the closed ball in $(\cT(a), d)$ with radius $r$ and center $\sigma$.

  Let us fix $\sigma \in \cT(a)$ and let us denote by $\tilde{\sigma} $ the unique point in 
$\lgeo \rho , \sigma \rgeo$ such that $d( \sigma , \tilde{\sigma})= a-b$. Observe that $\bar{B}_a (\sigma , 2(a-b))$ is the union of the $\cT_{^j}^{_b}(a-b)$ such that  $\sigma_{^j}^{_b}= \tilde{\sigma}$. Since the $\cT_{^j}^{_b}(a-b)$s are pairwise disjoints, we get 
$$ \forall a>b >0 , \; N_\gamma \,^{_{_-}}{\rm a.e.} \; \,   \forall \sigma \in \cT(a) , \; \cH( \bar{B}_a (\sigma , 2(a\!-\!b)) \, )= c_0 \ell^a (\bar{B}_a (\sigma , 2(a\!-\!b))\, ).$$
Consequently, there exists a Borel subset $A\subset \bbD ([0, \infty) , \bR)$ whose complementary set is $N_\gamma$-negligible and such that on $A$, one has $\cH_g (\bar{B}_a (\sigma , r))= c_0 \ell^a (\bar{B}_a (\sigma , r) ) <\infty$, for any $\sigma \in \cT(a)$ and any $r \in \bQ_+$, which easily entails (\ref{equaloclv}). \qed

\medskip

We have proved that (\ref{hypabsulv}) implies that there exists $c_0 \in (0, \infty)$ such that (\ref{equaloclv}) holds true. Let us furthermore assume that 
\begin{equation}
\label{testlvre}
\sum_{^{n \geq 1}} \frac{2^{-n}}{g(2^{-n})^{\gamma-1}} \; < \infty  \; . 
\end{equation}
Then, Proposition \ref{Htestlv} $(i)$ implies that $\cH_g (\cT(a))= \infty$,  
$N_\gamma$-a.e.$\;$on $\{ \cT(a) \neq \emptyset \}$, which contradict (\ref{equaloclv}) since 
$N_\gamma  (\ell^a(\cT(a)) = \infty)= 0$. Consequently (\ref{testlvre}) fails and Proposition \ref{Htestlv} $(ii)$ entails that $\ell^a(E)= \ell^a( \cT(a))$, where $E$ is a Borel subset of $\cT(a)$ such that $\limsup_{n} \ell^a(B( \sigma, 2^{-n}))/ g(2^{-n})= \infty$ for any $\sigma \in E$. By the comparison lemma for Hausdorf measures (Lemma \ref{genedens} $(iv)$), we get $\cH_g (E)= 0$, which contradicts (\ref{equaloclv}). This implies that (\ref{hypabsulv}) is false, which proves (Claim 2). This completes the proof of Theorem \ref{noreguHauslv}. \qed

\subparagraph{Proof of Theorem \ref{noreguHausmass}.}  
Let $g(r)= r^q s(r)$ where $q$ is nonnegative and $s$ is slowly varying at $0$. Recall that we furthermore assume that $g$ is a regular gauge function. Recall that $N_\gamma$-a.e.$\;$the Hausdorff dimension of $\cT$ is $\gamma/(\gamma-1)$. Thus, if $q > \gamma/(\gamma -1)$, then $\cH_g (\cT  )= 0$, $N_\gamma$-a.e.$\;$and if $q < \gamma/(\gamma -1)$, then $\cH_g (\cT  )= \infty$, $N_\gamma$-a.e. {\it We next restrict our attention to the case $q= \gamma /(\gamma-1)$}.

\medskip

{\it The general idea of the proof of Theorem \ref{noreguHausmass}} is the following: if $N_\gamma (\, 0< \cH_g ( \cT ) < \infty ) >0$, then we first prove that $ 0 <\cH_g ( \cT)< \infty $, $N_\gamma$-a.e. We next observe that $\cH_g $ behaves like $\bm$ with respect to the scaling property and the branching property and we prove it entails that $\cH_g = c_0 \bm$, where $c_0\in (0, \infty)$. Finally, we get a contradiction thanks to the test stated in Proposition \ref{Htestmass}. 

\medskip

We first need to state two preliminary results. We agree on the convention $\exp (-\infty )= 0$ and for any $a, \lambda >0$ and any $\mu \geq 0$, we set 
$$ \tilde{\kappa}_a (\lambda, \mu)= N_\gamma \big(  1-e^{-\lambda \cH_g (\bar{B} (\rho, a) \, ) -\mu L^a_\zeta       } \,  \big) \; \in [0, \infty] \; .$$
Let us fix $b \in (0, a )$. Recall that $(g_{^j}^{_b}, d_{^j}^{_b})$, $i \in \cI_b$ stands for the connected components of the open set $\{ t \geq 0: H_t >b \}$ and recall that for any $j \in \cI_b$, we denote by $H^{_{b, j}}$ the corresponding excursion of $H$ above $b$. We also set $\cT_{^j}^{_b}= p([g_{^j}^{_b}, d_{^j}^{_b}])$ and $\sigma_{^j}^{_b} = p(g_{^j}^{_b})$. Then, the subtree $(\cT_{^j}^{_b}, d, \sigma_{^j}^{_b})$ is isometric to the rooted compact real tree coded by the excursion $H^{_{b,j}}$. We set $\cT_{^j}^{_b} (\cdot \leq a-b)= \{ \sigma \in \cT_{^j}^{_b} : d(\sigma_{^j}^{_b}, \sigma) \leq a-b \}$ that is the closed ball of $(\cT_{^j}^{_b} , d, \sigma_{^j}^{_b})$ with center $\sigma_{^j}^{_b}$ and radius $a-b$. For any integer $n \geq 1$,  for any $\lambda \in (0, \infty)$ and for any $\mu \in [0, \infty)$, we then set 
\begin{eqnarray*}
 K_n = n \wedge \cH_g (\bar{B} (\rho, b) \,)  &+ &  \lambda \sum_{^{j \in \cI_b}} \un_{ \{ \sup H^{b, j}   > 1/n \}}  \, n \wedge \cH_g (\cT^b_j (\cdot \leq a-b) \, ) \;  \\
&+ &  \mu \sum_{^{j \in \cI_b}} \un_{ \{ \sup H^{b, j}   > 1/n \}}  \,  L^{a-b}_{\zeta_j } (j)  \; .
  \end{eqnarray*}
Here $L^{_{ a-b}}_{^{\zeta_j }} (j)$ stands for the local time of $H^{j,b}$ at level $a-b$. Since $H$ is continuous, the sum only contains a finite number of non-zero terms. Recall that 
since $\cH_g ( \cT_{^j}^{_b} (\cdot \leq a-b))$ is a measurable function of $H^{_{b,j}}$, the branching property (\ref{branchprop}) implies that for any $\lambda  \in (0, \infty)$, 
$$N_\gamma^{_{_{(b)}}} \,^{_{_{-}}}{\rm a.e.} \quad N_\gamma^{_{_{(b)}} } \big( e^{- K_n} \, \big| \, \cG_b \big)= 
e^{ -\lambda . \,  n\wedge \cH_g (\bar{B} (\rho , b) \, ) -L^b_\zeta  \tilde{\kappa}^{(n)}_{a-b}( \lambda , \mu) } \; ,$$
where 
$$\tilde{\kappa}^{(n)}_{a-b}( \lambda , 0)=N_\gamma \big[  \un_{\{ \sup H > 1/n\}} ( 1-e^{-\lambda . \, n \wedge \cH_g (\bar{B} (\rho, a-b)  ) -\mu L^{a-b}_\zeta}\,  \big) \big]  \; .$$
Monotone convergence implies $\lim_{n} \uparrow   \tilde{\kappa}^{_{(n)}}_{^{a-b}}( \lambda , 0)  =  \tilde{\kappa}_{^{a-b}}( \lambda , 0)  $. Then note that $N_\gamma$-a.e. 
$$\lim_{n} \uparrow K_{n} = \lambda \cH_g (\bar{B} (\rho, a) \,) + \mu L^a_\zeta \; .$$
The conditional dominated convergence theorem implies that for any $\lambda  \in (0, \infty)$, any $\mu \in [0, \infty)$ and any $a>b >0$, we $N_\gamma^{_{_{(b)}}} $-a.s.$\;$have 
\begin{eqnarray}  
\label{branchHausmass}
\hspace{-7mm}N_\gamma^{_{_{(b)}} } \big( \un_{\{  \cH_g (\bar{B} (\rho , a) ) < \infty  \}}e^{-\lambda \cH_g (\bar{B} (\rho , a)  ) -\mu L^a_\zeta  }  \big|  \cG_b \big) & = &  \nonumber \\ 
& & \hspace{-20mm} \un_{\{  \cH_g (\bar{B} (\rho , b)  ) < \infty  \}}e^{ -\lambda \cH_g (\bar{B} (\rho , b)  ) -L^b_\zeta  \tilde{\kappa}_{a-b}( \lambda , \mu ) }.
\end{eqnarray}
Arguing as in the proof of (\ref{scaleHauslv}), the scaling property of $\cT$ and the fact that $g$ is regularly varying at $0$ with exponent $\gamma / (\gamma-1)$ imply that for any $c \in (0, \infty)$,  the joint law of $c^{_{\gamma/(\gamma-1)}} \cH_g (\bar{B} (\rho , b/c))$ and 
$c^{_{1/(\gamma-1)}} L^{_{b/c}}_{^\zeta} $under $N_\gamma $ is the same as the joint law of $ \cH_g (\bar{B} (\rho , b))$ and 
$ L^{_{b}}_{^\zeta} $ under $c^{_{1/(\gamma-1)}} N_\gamma$. This easily entails 
\begin{equation}
\label{scaleHausmass}
\tilde{\kappa}_b ( \lambda, \mu )= c^{-\frac{1}{\gamma-1}} \tilde{\kappa}_{b/c} \big(  c^{\frac{\gamma }{\gamma-1}}\lambda, c^{\frac{1}{\gamma-1}}\mu  \big) \; , \quad b, \lambda , c , \mu >0 \; .
\end{equation}

\medskip

Now observe that Theorem \ref{noreguHausmass} is implied by the two following claims. 
\begin{description}  
\item[{\bf (Claim 1)}]  \hspace{-3mm} If $\, N_\gamma (\cH_g (\cT) <\infty ) >0$, then $N_\gamma (\cH_g (\cT) =\infty ) =0$. 

\medskip

\item[{\bf (Claim 2)}]   \hspace{-3mm} $N_\gamma (0 < \cH_g (\cT) < \infty) = 0$. 
\end{description}
We first prove (Claim 1). Let us suppose that $N_\gamma (\cH_g (\cT) < \infty) >0$. Then, there exists $a_0 >b_0 >0$ such that $N_\gamma ( \cH_g (\bar{B} (\rho , a_0) < \infty \, ; \, \sup H > b_0 ) >0$. Next, observe that the left member in (\ref{branchHausmass}) with $a= a_0$ and $b= b_0$ is strictly positive,  which entails that $\tilde{\kappa}_{a_0-b_0}( \lambda , 0)  < \infty$, for any $\lambda \in (0, \infty)$,  since we $N^{_{_{(b)}}}_\gamma$-a.s.$\;$have $L^{_b}_{^\zeta} >0$. Therefore, we have  
$$N_\gamma (    \cH_g (\bar{B} (\rho , a_0-b_0) ) = \infty)  \leq \tilde{\kappa}_{^{a_0-b_0}}( \lambda , 0)  < \infty \; , \quad \lambda \in (0, \infty) \; .$$ 
The scaling property easily entails that 
$$N_\gamma (\cH_g (\bar{B} (\rho , b) )= \infty) = b^{_{-1/(\gamma-1)}}N_\gamma ( \cH_g (\bar{B} (\rho , 1) )= \infty) \; .$$
 Since $b \mapsto N_\gamma (\cH_g (\bar{B} (\rho , b) )= \infty )$ is obviously non-decreasing and finite at $b= a_0-b_0$, we get $N_\gamma (\cH_g (\bar{B} (\rho , b) )= \infty)= 0$ for any $b>0$. Consequently, 
$N_\gamma (\cH_g (\cT)= \infty) = 0$, since $\cT$ is bounded. This completes the proof of the first claim.

\medskip

We now prove (Claim 2). We argue by contradiction, so we suppose that 
\begin{equation}
\label{hypabsumass}
N_\gamma ( 0 < \cH_g (\cT\ )< \infty) >0 \; .
\end{equation}
First, observe that (Claim 1) entails that $N_\gamma (\cH_g (\bar{B} (\rho , a)) = \infty)= 0$ for any $a >0$. Since $\cT$ is bounded, we get 
\begin{equation}
\label{Hausfiniee}
N_\gamma \,^{_{_-}}{\rm a.e.} \quad \cH_g (\cT) < \infty \; .
\end{equation}
Then observe that for any $b \in (0, \infty)$ the left member in (\ref{branchHausmass}) is strictly positive for any $\lambda >0$. This entails $\tilde{\kappa}_{^{a-b}}( \lambda , 0)  < \infty$ for any $a>b>0$, $\lambda >0$, since we $N^{_{_{(b)}}}_\gamma$-a.s.$\;$have $L^{_b}_{^\zeta} >0$, as already mentionned. More simply, we have proved 
$$\tilde{\kappa}_{{a}}( \lambda , 0)  < \infty \; , \quad a, \lambda \in (0, \infty) \; .$$
Let $\mu >0$. Observe that 
\begin{equation}
\label{incrkappamu}
0 \leq \tilde{\kappa}_{a}( \lambda , \mu) -\tilde{\kappa}_{a}( \lambda , 0) \leq \mu N_\gamma (L^{_a}_{^\zeta})= \mu \; .
\end{equation}
This implies that  $ \tilde{\kappa}_{a}( \lambda , \mu) <\infty $. Moreover (\ref{hypabsumass}) implies that $\tilde{\kappa}_{a} (\lambda, \mu)  >0$. Thus, we have proved that 
$$ \tilde{\kappa}_{a}( \lambda , \mu) \, \in (0, \infty) \; , \quad a, \lambda >0 \; , \; \mu \geq 0 \; .$$
Since $N_\gamma (L^{_b}_{^\zeta} \neq 0)= N_\gamma (\sup H >b)$,
(\ref{branchHausmass}) easily entails 
\begin{equation}  
\label{Hausflotmass}
\tilde{\kappa}_a ( \lambda, \mu )=  N_\gamma \big(1- e^{ -\lambda \cH_g (\bar{B}(\rho , b) \, ) -L^b_\zeta  \tilde{\kappa}_{a-b}( \lambda , \mu) } \big)= \tilde{\kappa}_b( \lambda, \tilde{\kappa}_{a-b}( \lambda, \mu) \, ).
\end{equation}

  We next prove that for any $\lambda \in (0, \infty)$, there exists $\phi (\lambda) \in (0, \infty)$, such that for any $a >0$, we have  
\begin{equation}  
\label{ordrun}  
 \tilde{\kappa}_a ( \lambda, 0)= N_\gamma \big( 1-e^{-\lambda \cH_g (\bar{B} (\rho , a) \, ) }  \big) = \phi(\lambda) \int_0^a \!\! N_\gamma \big( e^{-\lambda \cH_g (\bar{B} (\rho , b)\, ) } L^b_\zeta \big) \, db  .
\end{equation}
{\it Proof of (\ref{ordrun}):} we first set  $U( \lambda , a)=N_\gamma \big( e^{-\lambda \cH_g (\bar{B} (\rho , a)\, ) } L^a_\zeta \big)$, for any $a, \lambda \in (0, \infty)$. Observe that (\ref{branchHausmass}), 
(\ref{Hausfiniee}) and (\ref{incrkappamu}) imply for any $a>b>0$, and any $\mu \in (0, \infty)$
\begin{eqnarray*}
N_\gamma^{_{_{(b)}} } \big(e^{-\lambda \cH_g (\bar{B} (\rho , a)  )} \frac{_1}{^{\mu}} \big( 1-e^{ -\mu L^a_\zeta  }  \big) \big|  \cG_b \big) & = &  \nonumber \\ 
& & \hspace{-25mm} e^{ -\lambda \cH_g (\bar{B} (\rho , b)  )  -L^b_\zeta  \tilde{\kappa}_{a-b}( \lambda , 0 )   }  \frac{_1}{^{\mu}}\big( 1-e^{ -L^b_\zeta  (\tilde{\kappa}_{a-b}( \lambda , \mu )-\tilde{\kappa}_{a-b}( \lambda , 0 ) \, )} \big)\\
& & \hspace{-29mm} \leq   e^{ -\lambda \cH_g (\bar{B} (\rho , b)  )}L^b_\zeta \;  \frac{_1}{^{\mu}}  (\tilde{\kappa}_{a-b}( \lambda , \mu )-\tilde{\kappa}_{a-b}( \lambda , 0 ) \, ) \\
& & \hspace{-29mm} \leq  e^{ -\lambda \cH_g (\bar{B} (\rho , b)  )}L^b_\zeta \;  .
\end{eqnarray*}
Observe that the following limit is non-decreasing: $\lim_{\mu \downarrow 0} \uparrow  
\frac{1}{\mu}\big( 1-e^{ -\mu L^a_\zeta  }  \big)= L^a_\zeta$. Conditional monotone convergence entails 
$$ N_\gamma^{_{_{(b)}} } \big(e^{-\lambda \cH_g (\bar{B} (\rho , a)  )}  L^a_\zeta  \big|  \cG_b \big)
 \leq e^{ -\lambda \cH_g (\bar{B} (\rho , b)  )}L^b_\zeta . $$
We integrate this inequality with respect to $ N_\gamma^{_{_{(b)}} }$. Since $\{ L^a_\zeta  >0 \} \subset \{ L^b_\zeta >0 \}$, we easily get for any $ a>b>0$, and any $\lambda \in (0, \infty)$, 
\begin{equation}
\label{decreasizebias}
 U(\lambda , a) =  N_\gamma \big(e^{-\lambda \cH_g (\bar{B} (\rho , a)  )}  L^a_\zeta   \big) \leq  N_\gamma \big(e^{-\lambda \cH_g (\bar{B} (\rho , b)  )}  L^b_\zeta   \big) = U(\lambda , b) , 
\end{equation}

  Next, let $b, h \in (0, \infty)$ and note that $\cH_g (\bar{B} (\rho , b+h)\, ) -\cH_g (\bar{B} (\rho , b)\, ) = 0$, on the event $\{ \cT(b) = \emptyset \}$. Since $N_\gamma (\cT (b) \neq \emptyset )= N_\gamma (\sup H >b)$, 
(\ref{branchHausmass}) and an elementary inequality entails 
\begin{eqnarray*}
N_\gamma \big( e^{-\lambda \cH_g (\bar{B} (\rho , b) ) } \! -e^{-\lambda \cH_g (\bar{B} (\rho , b+h)) } \big)  &= & N_\gamma \big( e^{-\lambda \cH_g (\bar{B} (\rho , b) ) } \big( 1\!-\!e^{-L^b_\zeta \tilde{\kappa}_h ( \lambda, 0)} \big) \, \big) \\
 & \leq  & \tilde{\kappa}_h ( \lambda, 0)  
\end{eqnarray*}
(here we use the fact $N_\gamma (L^{_b}_{^\zeta})=  1$ in the last inequality). This implies 
\begin{equation}
\label{telescope}
 \tilde{\kappa}_a ( \lambda, 0)=  \tilde{\kappa}_{\frac{1}{n}} (\lambda, 0) + n \,    \tilde{\kappa}_{\frac{1}{n}} (\lambda, 0)\int_0^{\frac{\lfloor na \rfloor -1}{n} } \!\!\!\!\! G_n( b, \lambda) db + R_n (a, \lambda) \; , 
\end{equation}
where we have set 
\begin{displaymath}
\left\{ \begin{array}{l}
G_n(b, \lambda)=(\tilde{\kappa}_{1/n} ( \lambda, 0)\, )^{-1}  N_\gamma \big( e^{-\lambda \cH_g (\bar{B} (\rho , \lceil nb \rceil/n) ) } \big( 1\!-\!e^{-L^{_{\lceil nb \rceil/n}}_{^\zeta} \tilde{\kappa}_{1/n} ( \lambda, 0)} \big) \, \big),  \\
R_n (a, \lambda)=     N_\gamma \big( e^{-\lambda \cH_g (\bar{B} (\rho , \lfloor na \rfloor/n) ) } \big( 1\!-\!e^{-L^{_{\lfloor na \rfloor/n}}_{^\zeta} \tilde{\kappa}_{\{ an\}/n} ( \lambda, 0)} \big) \, \big) ,
\end{array} \right.
  \end{displaymath}
where  $\lfloor \cdot \rfloor$ stands for the integer-part function, where $\lceil \cdot \rceil= \lfloor \cdot \rfloor+1$ and where $\{ an\}= na-\lfloor na \rfloor $ stands for the fractional part of $na$. 

    Observe that $N_\gamma$-a.e.$\;$we have $\lim_{h \rightarrow 0} \cH_g ( \bar{B} (\rho , h))= \cH_g( \{ \rho \})= 0$, since any Hausdorf measure is diffuse. Dominated convergence entails that $\lim_{h \rightarrow 0} \tilde{\kappa}_{h} (\lambda, 0)= 0$, for any $\lambda \in (0, \infty)$. 
Consequently, $\lim_n  \tilde{\kappa}_{1/n} (\lambda , 0)= 0$. Moreover, we get 
$$ R_n (a, \lambda) \leq \tilde{\kappa}_{\frac{_{\{ an\} }}{^{n}}} ( \lambda, 0) \;   N_\gamma \big( L^{_{\lfloor na \rfloor/n}}_{^\zeta} \big) =\tilde{\kappa}_{\frac{_{\{ an\} }}{^{n}}} ( \lambda, 0)  \underset{n \rightarrow \infty}{-\!\!\!-\!\!\! \!\longrightarrow} \; 0 . $$
Next, observe that 
$$G_n (b, \lambda) \leq  N_\gamma \big( e^{-\lambda \cH_g (\bar{B} (\rho , \lceil nb \rceil/n) ) } L^{_{\lceil nb \rceil/n}}_{^\zeta} \, \big)= U(\lambda , \frac{_{\lceil nb \rceil}}{^{n}} ) \; .$$
By (\ref{decreasizebias}), we get $G_n (b, \lambda) \leq U(\lambda , b)$. 
By (\ref{Hausfiniee}), $\cH_g$ is a finite measure. Then, $b \mapsto  \cH_g (\bar{B} (\rho , b)\, ) $ is right continuous. Recall that we work with a right-continuous modification of $b \mapsto L^{_b}_{^\zeta}$. Therefore, Fatou's Lemma implies 
$$ U (\lambda, b)= N_\gamma \big( e^{-\lambda \cH_g (\bar{B} (\rho , b)\, ) } L^b_\zeta \big)  \leq \liminf_n G(b , \lambda) \; ,$$
Thus, we get 
$$ \lim_n G_n (b, \lambda)= U(b, \lambda) = N_\gamma \big( e^{-\lambda \cH_g (\bar{B} (\rho , b)\, ) } L^b_\zeta \big) \; .$$
By dominated convergence, we get 
$$ \lim_n   \int_0^{\frac{\lfloor na \rfloor -1}{n}  } \!\!\!\!\! G_n( b, \lambda) db= \int_0^a \!\! N_\gamma \big( e^{-\lambda \cH_g (\bar{B} (\rho , b)\, ) } L^b_\zeta \big) \, db .$$
This limit combined with (\ref{telescope}) implies $\lim_n n   \tilde{\kappa}_{\frac{1}{n}} (\lambda, 0)=\phi(\lambda) \in (0, \infty)$, which implies (\ref{ordrun}).   \cq

\medskip

Next observe that (\ref{ordrun}) implies $(\partial \tilde{\kappa}_a / \partial a) (\lambda, 0)= \phi(\lambda) \partial (\tilde{\kappa}_a / \partial \mu) (\lambda, 0)$ and by the scaling property (\ref{scaleHausmass}) we easily get 
\begin{equation}
\label{philin}
\phi (\lambda)= c_0 \lambda \, , \quad {\rm with} \quad  c_0:= \phi(1) \in (0, \infty)\; .
\end{equation}  

\medskip

We next  prove that for any $b \in (0, \infty)$, 
\begin{equation}
\label{coincidmass}
N_\gamma\,^{_{_-}}{\rm a.e.}  \quad \cH_g ( \bar{B} (\rho, b)\, ) = \bm   ( \bar{B} (\rho, b)\, ) \; .
\end{equation}
{\it Proof of (\ref{coincidmass}):} we first prove this result for $b= 1$. To that end, we set 
$$ D_n (\lambda)=   1\!-\!e^{-\lambda \cH_g (\bar{B} (\rho , 1) ) } \!- \!S_n (\lambda) $$
where $S_n (\lambda)$ stands for 
$$S_n (\lambda)= \sum_{^{1 \leq k < n}}  e^{-\lambda \cH_g (\bar{B} (\rho , k/n) ) } \big(1-e^{-L^{k/n}_\zeta \tilde{\kappa}_{1/n} (\lambda, 0) } \big) . $$
We want to prove that $\lim_n N_\gamma (D_n (\lambda)^2)= 0$. To that end, observe that 
$$ D_n (\lambda)= \sum_{^{0 \leq k < n}} V_k \; , $$
where we have set 
\begin{displaymath}
\left\{ \begin{array}{l}
 V_k= e^{-\lambda \cH_g (\bar{B} (\rho , k/n) ) } \big(e^{-L^{k/n}_\zeta \tilde{\kappa}_{1/n} (\lambda, 0) } - 
 e^{-\lambda \cH_g ( C(k/n)\, ) } \big)\; , \quad 1 \leq k < n \; , \\
 V_0 =  1-e^{-\lambda \cH_g (\bar{B} (\rho , 1/n) ) } \; , 
\end{array} \right.
\end{displaymath}
with $C(k/n)= \bar{B} (\rho , \frac{{k+1}}{n}) ) \backslash \bar{B} (\rho , \frac{{k}}{n}) $. Let $1 \leq k \leq n-1$. Observe that $\cT( k /n)= \emptyset$, $N_\gamma$-a.e.$\;$on $\{ \sup H \leq k/n\}$. Thus, $V_k= 0$, $N_\gamma$-a.e.$\;$on $\{  \sup H \leq k/n \}$. This easily implies that 
$$ \forall \, 0 < k < n \, , \quad | V_k| \leq 2 \cdot  \un_{\{ \sup H > k/n \}} \; .$$
Since $N_\gamma ( \sup H >k/n)= v(k/n) < \infty$, we have $N_\gamma (|V_k V_\ell | )< \infty$, for any $0 \leq k < \ell <n$ and $N_\gamma (V_k^2) < \infty$, for any $1 \leq k <n$. Next observe that $V_0^2 \leq V_0$. Thus $N_\gamma (V_0^2 )\leq N_\gamma (V_0)= \tilde{\kappa}_{1/n} (\lambda, 0) < \infty$. 
These estimates justify the following: 
\begin{equation}
\label{devel}
N_\gamma (D_n (\lambda)^2)= 2 \sum_{^{0 \leq k <\ell <n }} N_\gamma  (V_kV_\ell) + \sum_{^{0 \leq k <n }} N_\gamma ( V_k^2) \; , 
\end{equation}
We first fix $0 \leq k < \ell < n$. Observe that $V_\ell $ is $N^{_{_{(\ell/n)}}}_\gamma$-integrable and (\ref{branchHausmass}) implies $N^{_{_{(\ell/n)}}}_\gamma ( V_\ell \, | \cG_{^{\ell/n}}) = 0$. Moreover, $V_k$ is $\cG_{^{(k+1)/n}}$-measurable. Therefore, it is $\cG_{^{\ell /n}}$-measurable. Since $N_{\gamma} ( \cT(\ell/n) \neq \emptyset )= N_\gamma ( \sup H > \ell/n) = v(\ell/n)$, we  get 
$$N_\gamma \big( V_kV_\ell  \big)=  v\big(\frac{_\ell}{^n} \big) \, N^{_{_{(\ell/n)}}}_\gamma \! \big(V_kV_\ell \big)= v\big(\frac{_\ell}{^n} \big)\, N^{_{_{(\ell/n)}}}_\gamma \! \big(V_k \, N^{_{^{\ell/n}}}_\gamma \!( V_\ell \, | \,  \cG_{\ell/n})\, \big)= 0 \; .$$
We next fix $1 \leq k < n$. An easy argument combined with (\ref{branchHausmass}) entails that for any $1 \leq k < n$, we have 
$$ N^{_{^{k /n}}}_\gamma \! \big(V^2_k \, | \, \cG_{k/n} \big)=  e^{-2\lambda \cH_g (\bar{B} (\rho , k/n) ) } \big(e^{-L^{k/n}_\zeta \tilde{\kappa}_{1/n} (2\lambda, 0) } - 
 e^{- 2L^{k/n}_\zeta \tilde{\kappa}_{1/n} (\lambda, 0) } \big) . $$
Note that $\lambda \mapsto \tilde{\kappa}_{a} (\lambda, 0)  $ is clearly concave. Thus, $ 2\tilde{\kappa}_{1/n} (\lambda, 0) -\tilde{\kappa}_{1/n} (2\lambda, 0)  $ is nonnegative, which implies 
$$ N^{_{_{(k/n)}}}_\gamma \! \big(V^2_k \, | \, \cG_{k/n} \big) \leq ( 2\tilde{\kappa}_{1/n} (\lambda, 0) -\tilde{\kappa}_{1/n} (2\lambda, 0) \, )L^{_{k/n}}_{^\zeta} .$$
This entails $N_\gamma (V_k^2) \leq 2\tilde{\kappa}_{1/n} (\lambda, 0) -\tilde{\kappa}_{1/n} (2\lambda, 0)$, since $N_\gamma (L^{_{k/n}}_{^\zeta})= 1 $. We also checks that $N_\gamma (V_0)= 2\tilde{\kappa}_{1/n} (\lambda, 0) -\tilde{\kappa}_{1/n} (2\lambda, 0)$. These bounds combined with (\ref{devel}) imply 
$$N_\gamma  (D_n (\lambda)^2) \leq 2n\tilde{\kappa}_{1/n} (\lambda, 0) -n\tilde{\kappa}_{1/n} (2\lambda, 0)\; .$$
When $n$ goes to $\infty$, the right member of the previous inequality tends to $2\phi(\lambda)-\phi( 2\lambda)$ that is null since $\phi$ is linear. Thus, 
$$ \forall \lambda \in (0, \infty) \; , \quad N_\gamma \big( D_n (\lambda)^2 \big) = 0 \; .$$
Fatou's lemma then entails $\liminf_n |D_n (\lambda)|= 0$ $N_\gamma$-a.e. Now observe that 
$$ S_n (\lambda)= \frac{_{n-1}}{^n} \, n \tilde{\kappa}_{\frac{_1}{^n}} (\lambda, 0)  \!\! \int_{1/n}^1 \!\!\! e^{-\lambda \cH_g (\bar{B} (\rho , \lfloor b/n \rfloor /n))} \frac{_{1}}{^{\tilde{\kappa}_{\frac{_1}{^n}} (\lambda, 0)}}  \big( 1-e^{L^{_{\lfloor b/n \rfloor /n}}_{^{\zeta}}   \tilde{\kappa}_{1/n} (\lambda, 0)  }\big) \, db . $$
Since $\lim_n n  \tilde{\kappa}_{1/n} (\lambda, 0) = \phi (\lambda)= c_0 \lambda$, we easily get 
$$ N_\gamma \,^{_{_-}}{\rm a.e.} \quad  \lim_n S_n (\lambda)= c_0 \lambda \int_0^1 e^{-\lambda \cH_g (\bar{B} (\rho , b))} L^b_\zeta \, db \; . $$
This implies that for any $\lambda \in (0, \infty)$, $N_\gamma$-a.e.
$$1-  e^{-\lambda \cH_g (\bar{B} (\rho , 1))}= c_0 \lambda \int_0^1 e^{-\lambda \cH_g (\bar{B} (\rho , b))} L^b_\zeta \, db \; . $$
Divide this equation by $\lambda$ and let $\lambda$ go to $0$, to get $\cH_g (\bar{B} (\rho , 1))= \int_{0}^{1} L^{_b}_{^\zeta} db$. Now recall that $\bm = \int_{0}^{\infty} \ell^b db$, which implies 
(\ref{coincidmass}) when $b= 1$.  

\smallskip

By the scaling property, the joint law of $\cH_g (\bar{B} (\rho , b))$ and $\bm (\bar{B} (\rho , b)) $ under $N_\gamma$ is the same as the joint law of 
$b^{_{\gamma/(\gamma-1)}}\cH_g (\bar{B} (\rho , 1))$ and $b^{_{\gamma/(\gamma-1)}} \bm (\bar{B} (\rho , 1)) $ under $b^{_{-1/(\gamma-1)}}N_\gamma$, which easily implies (\ref{coincidmass}) for any $b \in (0, \infty)$.  \cq 
 
\medskip

We next prove the following 
\begin{equation} 
 \label{equamass}
N_\gamma \,^{_{_-}}{\rm a.e.} \quad \cH_g = c_0 \,  \bm \; . 
\end{equation}
{\it Proof of (\ref{equamass}):} recall that $(g_{^j}^{_a}, d_{^j}^{_a})$, $j \in \cI_a$, stands for the excursion intervals of $H$ above $a$ and that $H^{_{a, j}}$ stands for the excursion correponding to  $(g_{^j}^{_a}, d_{^j}^{_a})$. Recall that we have set $\cT_{^j}^{_a}= p([g_{^j}^{_a}, d_{^j}^{_a}])$ and $\sigma_{^j}^{_a} = p(g_{^j}^{_a})$ so that the subtree $(\cT_{^j}^{_a}, d, \sigma_{^j}^{_a})$ is isometric to the rooted compact real tree coded by the excursion $H^{_{a,j}}$. For any $b \geq 0$, recall the notation 
$\cT_{^j}^{_a} (\cdot \leq b)= \{ \sigma \in \cT_{^j}^{_a}: d(\sigma_{^j}^{_a} , \sigma) \leq b) \}$ that is the closed ball in  $(\cT_{^j}^{_a}, d, \sigma_{^j}^{_a})$ with center $\sigma_{^j}^{_a}$ and radius $b$. Since $\cH_g (\cT_{^j}^{_a} (\cdot \leq b) \,)$ is a measurable function of $H^{_{a,j}}$, the branching property (\ref{branchprop}) and (\ref{coincidmass}) imply 
for any $a, b\geq 0$, 
\begin{equation}
\label{qzigh}
\forall \, a , b >0 , \;  N_\gamma \,^{_{_-}}{\rm a.e.} \; \forall j \in \cI_a , \quad \cH_g (  \cT^a_j ( \cdot \leq b) \, ) = c_0 \, \bm (  \cT^a_j (\cdot \leq b) \, ) . 
\end{equation}
Recall that $\bm ({\bf Sk} (\cT))= 0$, $N_\gamma$-a.e.$\;$(here ${\bf Sk} (\cT)$ stands for the skeleton of $\cT$). This result, combined with (\ref{qzigh}), shows that there exists a Borel set $A\subset \bbD ([0, \infty) , \bR)$ whose complementary set is 
$N_\gamma$-negligible and such that $\bm ({\bf Sk} (\cT))= 0$ and 
\begin{equation}
\label{Acontr}
\forall \,  a, b \in \bQ_+ \,  , \; \forall  j \in \cI_a \; , \quad \cH_g (  \cT^a_j ( \cdot \leq b) \, ) = c_0 \,  \bm (  \cT^a_j ( \cdot \leq b) \, )  
 \end{equation}
on $A$. {\it We now work deterministically on $A$}. Note that for any $a\in \bQ_+$, any $j \in \cI_a$ and any $b \in (0, \infty)$, we have 
$$ \cT_{^j}^{_a} ( \cdot \leq b)\, = \!\!\! \bigcap_{^{b^\prime \in \bQ_+: b^\prime \geq b}} \cT_{j}^{a} ( \cdot  \leq b^\prime) \; .$$
Since $\bm$ and $\cH_g$ are finite measures, (\ref{Acontr}) holds for any $b\in [0, \infty)$.

   Let $\sigma $ be a point in $\cT$ that is not a leaf. 
Then $\cT \backslash \{ \sigma \}$ has at least one connected component that does not contain the root $\rho$. Let us denote such a component by $\tilde{\cT}^o$. To ease the discussion, we call such a subset of $\cT$ an {\it open upper subtree}. Then, for any $b \in (0, \infty)$, we denote by $ \tilde{\cT}^o (\leq b)$ the set of $\sigma^\prime \in \tilde{\cT}^o$ such that $d(\sigma, \sigma^\prime) \leq b$. It is easy to prove that $ \tilde{\cT}^o (  \cdot \leq b)$ 
is the union of a non-decreasing sequence of subtrees of the form $ \cT^{_a}_{^j} ( \cdot \leq b^\prime)$, 
with $a\in \bQ_+$. This entails $\cH_g (\tilde{\cT}^o ( \cdot \leq b))= c_0 \,  \bm (\tilde{\cT}^o ( \cdot \leq b))$. 

 We next fix $\sigma \in \cT$ and $r >0$. We denote by $\cT^{_o}_{^j}$, $j \in \cJ$ the connected components 
of $\cT \backslash \lgeo \rho , \sigma \rgeo$. For any $j \in \cJ$, denote by $\sigma_j$ the unique point of $\lgeo \rho , \sigma \rgeo$ such that $\{ \sigma_j \} \cup \cT^{_o}_{^j}$ is the closure of $\cT^{_o}_{^j}$. 
Note that $\cT_j $ is a connected component of $\cT \backslash \{ \sigma_j \}$ that does not contains the root. Namely, $\cT^{_o}_{^j}$ is an open upper subtree. Moreover, observe that  
$$ \bar{B} (\rho , r) \backslash \lgeo \rho , \sigma \rgeo= \bigcup \big\{ \cT_j^o ( \,  \cdot \,   \leq\! r\!-\!d(\sigma, \sigma_j) \,  ) \, ;  \quad   j \in \cJ : 0 \leq d(\sigma, \sigma_j) < r  \big\} . $$
This implies $\cH(  \bar{B} (\rho , r) \backslash \lgeo \rho , \sigma \rgeo) = c_0 \, \bm (  \bar{B} (\rho , r) \backslash \lgeo \rho , \sigma \rgeo)$. Now note that $\lgeo \rho , \sigma \rgeo$ is isometric to a compact interval of the line. Thus, the Hausdorff dimension of $\lgeo \rho , \sigma \rgeo$ is $1$ (or $0$ if it reduces to $\{ \rho \}$). Therefore, $\cH_g (\lgeo \rho , \sigma \rgeo)= 0$, since $g$ is regularly varying at $0$ with exponent $\gamma/ (\gamma-1) >1$. Next observe that $\lgeo \rho , \sigma \lgeo \subset {\bf Sk} (\cT)$. Consequently, we get $\bm (\lgeo \rho , \sigma \rgeo)= 0$. We thus have proved that on $A$, $\cH_g$ and $ c_0 \bm$ are finite Borel measures on $\cT$ that agree on the set of all closed balls of $\cT$. This clearly implies (\ref{equamass}).  \cq

 \medskip

We have proved that (\ref{hypabsumass}) implies $\cH_g= c_0 \, \bm$ $N_\gamma$-a.e. We now argue as in the proof of Theorem \ref{noreguHauslv} to get a contradiction thanks to the test stated in Proposition \ref{Htestmass}. This proves that (\ref{hypabsumass}) is wrong, which entails (Claim 2). As already mentioned, it completes the proof of Theorem \ref{noreguHausmass}. \qed



\end{document}